%% file: BridWiltInvent3.tex
\documentclass[12pt]{amsart} 

\usepackage[psamsfonts]{amssymb} 
\usepackage{amsfonts,amsmath,mathabx}
\usepackage{color}

\usepackage{amsthm, amssymb}
\usepackage{amsfonts} 
\usepackage{epsfig,multicol} 

\usepackage{xypic}
\usepackage{hyperref}
\usepackage{psfrag}

\input xy
\xyoption{all}

\newtheorem{thmA}{Theorem}

\newtheorem{theorem}{Theorem}[section]

\newtheorem{proposition}[theorem]{Proposition}
\newtheorem{lemma}[theorem]{Lemma}
\newtheorem{corollary}[theorem] {Corollary}

\newtheorem{conjecture}[theorem]{Conjecture}

\theoremstyle{remark}
\newtheorem{remark}[theorem]{Remark}
\newtheorem{remarks}[theorem]{Remarks}
\newtheorem{example}[theorem]{Example}

\theoremstyle{definition}
\newtheorem{definition}[theorem]{Definition}

\def\Z{\mathbb Z}
\def\N{\mathbb N}

\def\G{\Gamma}

\def\-{\overline}
\def\wh{\widehat}

\def\rk{{\rm{rk}}}

\def\G{\Gamma}  
\def\g{\gamma}
\def\k{\kappa}

\def\<{\langle}
\def\>{\rangle}

\title[The Profinite Triviality Problem]{The triviality problem for profinite completions}

\author[Martin R.\ Bridson]{Martin R.\ Bridson}
\address{Mathematical Institute, Andrews Wiles Building, Radcliffe Observatory Quarter, Oxford OX2 6GG, UK}
\email{bridson@maths.ox.ac.uk}
 
\author[Henry Wilton]{Henry Wilton}
\address{Department of Mathematics, University College London, Gower Street, London WC1E  6BT, UK}
\curraddr{DPMMS, Centre for Mathematical Sciences, Wilberforce Road, Cambridge CB3 0WB, UK}
\email{h.wilton@maths.cam.ac.uk}

\thanks{This work was supported by Fellowships from the EPSRC (both authors) and by a Wolfson Research Merit Award from the Royal Society (first author).}
 
 \date{15 December 2014}


\begin{document}

\maketitle

\begin{abstract}
We prove that there is no algorithm that can determine whether or not a finitely presented group has a non-trivial finite quotient; indeed, this property remains undecidable among the fundamental groups of compact, non-positively curved square complexes.  We deduce that many other properties of  groups are undecidable. For hyperbolic groups, there cannot exist algorithms to determine largeness, the existence of a linear representation with infinite image (over any infinite field), or the  rank of the profinite completion.
\end{abstract}


\section{Introduction}
The basic decision problems for finitely presented groups provided a guiding theme for combinatorial and geometric group theory throughout the twentieth century. Activity in the first half of the century was framed by Dehn's articulation of the core problems in 1911 \cite{dehn_unendliche_1911}, and it reached a climax in 1957-58 with the proof by Novikov \cite{novikov_algorithmic_1955} and Boone \cite{boone_word_1959} that there exist finitely presented groups with unsolvable word problem. In the wake of this, many other questions about general finitely presented groups were proved to be algorithmically unsolvable (cf.~Adyan \cite{adyan_algorithmic_1955,adyan_unsolvability_1957}, Rabin \cite{rabin_recursive_1958}, Baumslag--Boone--Neumann \cite{bausmlag_unsolvable_1959}). In the decades that followed, the study of decision problems shifted towards more refined questions concerning the existence of algorithms within specific classes of groups, and to connections with geometry and topology. However, certain basic decision problems about general finitely presented groups were not covered by the techniques developed in mid-century and did not succumb to the geometric techniques developed in the 1990s. The most obvious of these is the following: can one decide whether or not a group has a proper subgroup of finite index?
 
Our main purpose here is to settle this question.

\begin{thmA}\label{thm: Main theorem}\label{t:main}
There is no algorithm that can determine whether or not a finitely presented group has a proper subgroup of finite index. 
\end{thmA}

The technical meaning of this theorem is that there is a recursive sequence of finitely presented groups $G_n$ with the property that the set of natural numbers
\[
\{n\in\N\mid \exists H\subsetneq G_n,~|G_n/H|<\infty\}
\]
is recursively enumerable but not recursive.  More colloquially, it says that the problem of determining the existence of a proper subgroup of finite index is {\em undecidable}.

We shall strengthen Theorem \ref{t:main} by proving that the existence of such subgroups remains undecidable in classes of groups where other basic decision problems of group theory are decidable, such as biautomatic groups and the fundamental groups of compact, non-positively curved square complexes. We include this last refinement in the following geometric strengthening of Theorem \ref{t:main}.

\begin{thmA}\label{t:covers}
There is no algorithm that can determine if a compact square complex of non-positive curvature
has a non-trivial, connected, finite-sheeted covering. 
\end{thmA}

There are various other natural reformulations of Theorem \ref{t:main} (and its refinements), each creating a different emphasis. For emphasis alone, one could rephrase Theorem \ref{t:main}  as {\em ``the triviality problem for profinite completions of finitely presented groups is undecidable"}: there is no algorithm that, given a finitely presented group $G$, can decide whether the profinite completion $\widehat G$ is trivial.  More substantially, since all finite groups are linear (over any field) and linear groups are residually finite, we can rephrase our main result
 as follows:
 \smallskip

{\em There is no algorithm that can determine whether or not a finitely presented group has a non-trivial
finite-dimensional linear representation (over any field); indeed the existence of such a representation is
undecidable even for the fundamental groups of compact, non-positively curved square complexes.}

\smallskip

In Section \ref{s:slobo} we shall explain how classical work of Slobodskoi \cite{slobodskoi_undecidability_1981} on the universal theory of finite groups can be interpreted as a profinite analogue of the Novikov--Boone theorem: by definition, the profinite completion $\widehat G$ is the inverse limit of the finite quotients of $G$,  and the kernel of the natural homomorphism $G\to \wh{G}$ consists of precisely those $g\in G$  that have trivial image in every finite quotient of $G$; implicitly, Slobodskoi constructs  a finitely presented group $G$ in which there is no algorithm to determine which words in the generators represent such $g\in G$.  In the setting of discrete groups, one can parlay the undecidability of the word problem for a specific group into the undecidability of the triviality problem for finitely presented groups by performing a sequence of HNN extensions and amalgamated free products, as described in Section \ref{s:strategy}.  Although the profinite setting is  more subtle and does not allow such a direct translation, we will attack the triviality problem from a similar angle, deducing Theorem \ref{t:main} from Slobodskoi's construction and the following {\em Encoding Theorem}. This is the key technical result in this paper; its proof is significantly more complex than that of the corresponding theorem for discrete groups and the details are much harder.

\begin{thmA}[Encoding Theorem]\label{thm: Profinite groups from words}\label{t:tech}
There is an algorithm that takes as input a finite presentation $\<A\mid R\>$ for a group $G$ and a word $w\in F(A)$ and outputs a presentation for a finitely presented group $G_w$ such that
\[
\widehat{G}_w\cong 1\Leftrightarrow w=_{\widehat{G}} 1~.
\]
\end{thmA}

Theorems \ref{t:main} and \ref{t:tech} imply that various other properties of finitely presented groups cannot be determined algorithmically. The properties that we shall focus on, beginning with the property $\widehat G\cong 1$ itself, are neither Markov nor co-Markov, so their undecidability cannot be established using the Adyan--Rabin method.
\smallskip

Some of the most profound work in group theory in recent decades concerns the logical complexity of (word-)hyperbolic groups. In that context, one finds undecidability phenomena associated to finitely generated subgroups but the logical complexity of hyperbolic groups themselves is strikingly constrained (see, for instance, \cite{sela_diophantine_2009} and \cite{kharlampovich_decidability_2013}).  Nevertheless, we {\em conjecture} that there does not exist an algorithm that can determine if a hyperbolic group has a non-trivial finite quotient (Conjecture \ref{conj: Hyperbolic undecidability}). This conjecture would be false if hyperbolic groups were all residually finite. Indeed, we shall prove (Theorem \ref{t:mainHyp}) that this conjecture is equivalent to the assertion that there exist hyperbolic groups that are not residually finite.

We shall also prove that, as it stands, Theorem \ref{t:main} allows one to establish various new undecidability phenomena for hyperbolic groups.  We recall some definitions. The \emph{first betti number} $b_1(\Gamma)$ of a group $\G$ is the dimension of $H_1(\G,\mathbb{Q})$ and the \emph{virtual first betti number}  $vb_1(\Gamma)$ is the (possibly infinite) supremum of $b_1(K)$ over all subgroups $K$ of finite index in $\Gamma$.  A group is  {\em large} if it has a subgroup of finite index that maps onto a non-abelian free group.  Note that if $\Gamma$ is large then $vb_1(\Gamma)=\infty$.

The following theorem summarizes our undecidability results for hyperbolic groups.

\begin{thmA}\label{t:hyp}
There do not exist algorithms that, given a finite presentation of a torsion-free hyperbolic group $\G$, can determine:
\begin{enumerate}
\item \label{i: largeness} whether or not $\G$ is large;
\item \label{i: vb} for any $1\leq d\leq\infty$, whether or not $vb_1(\G)\geq d$;
\item \label{i: finite linear rep} whether or not every finite-dimensional linear representation of $\G$ has finite image;
\item \label{i: finite linear rep fixed k} for a fixed infinite field $k$, whether or not  every finite-dimensional representation of $\G$
over $k$ has finite image;
\item \label{i: profinite rank} whether or not, for any fixed $d_0> 2$, the profinite completion of $\G$ can be generated (topologically) by a set of cardinality less than $d_0$.
\end{enumerate}
\end{thmA}

Items (\ref{i: largeness}) and (\ref{i: vb}) are contained in Theorem \ref{thm: Hyperbolic largeness and vb_1}, items (\ref{i: finite linear rep}) and (\ref{i: finite linear rep fixed k}) are contained in Theorem \ref{thm: Hyperbolic linear representations}, and item (\ref{i: profinite rank}) is contained in Theorem \ref{thm: Hyperbolic profinite rank}.

We shall prove in Section \ref{s:profinite} that the profinite-rank problem described in item (\ref{i: profinite rank}) remains undecidable among residually-finite hyperbolic groups. In that  context, the bound $d_0>2$ is optimal, because the profinite rank of a residually finite group $\G$ is less than $2$ if and only if $\G$ is cyclic, and it is easy to determine if a hyperbolic group is cyclic.  Furthermore, Theorem \ref{t:mainHyp} tells us that for $d_0\le 2$, problem (\ref{i: profinite rank}) is decidable if and only if every hyperbolic group is residually finite. 
 
Item (\ref{i: largeness}) should be contrasted with the fact that there {\em does} exist an algorithm that can determine whether or not a finitely presented group maps onto a non-abelian free group: this is a consequence of Makanin's deep work on equations in free groups \cite{makanin_decidability_1984}.
 
Our final application is to the isomorphism problem for the profinite completions of groups.
The arguments required to deduce this from Theorem \ref{thm: Main theorem} are lengthy and somewhat technical, so we shall present them elsewhere \cite{bridson_isomorphism_2013}. 

\begin{thmA}
There are two recursive sequences of finite presentations for residually finite groups $A_n$ and $B_n$ together with monomorphisms $f_n:A_n\to B_n$ such that:
\begin{enumerate}
\item $\widehat{A}_n\cong\widehat{B}_n$ if and only if the induced map on profinite completions $\hat{f}_n$ is an isomorphism; and
\item the set $\{n\in\N\mid \widehat{A}_n\ncong\widehat{B}_n\}$
is recursively enumerable but not recursive.
\end{enumerate}
\end{thmA} 

This paper is organised as follows. In Section \ref{s:slobo} we explain what we need from Slobodskoi's work. In Section \ref{s:strategy} we lay out our strategy for proving Theorem \ref{t:tech}, establishing the notation to be used in subsequent sections and, more importantly, providing the reader with an overview that should sustain them through the technical arguments in Sections \ref{s:omni} and \ref{s: Malnormal subgroups}.  Theorem \ref{t:tech} is proved in Section \ref{s:proof} and, with Slobodskoi's construction in hand, Theorem \ref{t:main} follows immediately. Sections \ref{s: Malnormal subgroups} and \ref{s:proof} form the technical heart of the paper. Many of the arguments in these sections concern malnormality for subgroups of virtually free groups. The techniques here are largely topological, involving the careful construction of coverings of graphs (and, implicitly, graphs of finite groups) and the analysis of fibre products in the spirit of John Stallings \cite{stallings_topology_1983}.

 In Section \ref{s:npc} we prove that the existence of finite quotients remains undecidable in 
the class of non-positively curved square complexes. Section \ref{s:profinite} deals with profinite rank, 
and the remaining results about hyperbolic groups are proved in Section \ref{s:hyp}.

\subsection*{Acknowledgements.} We first tried to prove Theorem \ref{t:main} at the urging of Peter Cameron, who was interested in its implications for problems in combinatorics \cite{cameron_extending_2004,bridson_undecidability_2014}; we are grateful to him for this impetus. We thank Jack Button and Chuck Miller for stimulating conversations about Theorem \ref{t:main} and its consequences.  Finally, we are grateful for the insightful comments of the anonymous referee.

\section{Slobodskoi's theorem}\label{s:slobo}

In this section we explain how the following theorem is contained in 
Slobodskoi's work on the universal theory of finite groups
\cite{slobodskoi_undecidability_1981}. We write $F(A)$ to denote the free group on a set $A$.

\begin{theorem}\label{thm: Slobodskoi}
There exists a finitely presented group $G\cong\<A\mid R\>$ in which there is no algorithm to
decide which elements have trivial image in every finite quotient. More precisely,
the set of reduced words
\[
\{w\in F(A)\mid w\neq_{\widehat{G}} 1\}
\]
 is recursively enumerable but not recursive.
\end{theorem}

The theorem that Slobodskoi actually states in \cite{slobodskoi_undecidability_1981} is the following.

\begin{theorem}[\cite{slobodskoi_undecidability_1981}]\label{thm: Universal theory}
The universal theory of finite groups is undecidable.
\end{theorem}

Slobodskoi's proof of Theorem \ref{thm: Universal theory} is clear and explicit. It revolves around a finitely
presented group $G=\<a_1,\dots,a_n\mid r_1, \dots,r_m\>$ 
that encodes the workings of a 2-tape Minsky machine $M$ that computes a partially recursive function.  Associated to this machine one has a disjoint pair of subsets $S_0, S_1\subseteq \mathbb{N}$ (denoted $X$ and $Y$ in \cite{slobodskoi_undecidability_1981}) that are \emph{recursively inseparable}: $S_0$ is the set of natural numbers $k$
such that $M$ halts on input $2^k$ and $S_1$ is the set of $k$ such
that on input $2^k$ the machine $M$ visits the leftmost square of at least one of its tapes infinitely often. To say that they are
recursively inseparable means that there does not exist a recursive set $D\subseteq\N$ such that $S_1\subseteq D$ and $S_0\cap D=\emptyset$.

By means of a simple recursive rule, Slobodskoi defines two sequences of words $w_1^{(k)}, w_2^{(k)}\ (k\in\N)$ in the letters $A^{\pm 1}$. He then considers the following sentences in the first-order logic of groups.
$$
\Psi(k) \equiv \forall a_1,\dots,a_n [(r_1 \neq 1)\vee \dots \vee (r_m\neq 1)
\vee (w_1^{(k)}=w_2^{(k)}=1)]
$$
Note that the sentence $\Psi(k)$ is {\em{false}} in a group $\G$ if and only if
there is a homomorphism $\phi:G\to\G$ such that at least
one of $\phi(w_1^{(k)})$
or $\phi(w_2^{(k)})$ is non-trivial.  In particular, $\Psi(k)$ is false in some finite group $\G$ if and only if either $w_1^{(k)}\neq_{\widehat{G}}1$ or $w_2^{(k)}\neq_{\widehat{G}}1$.

Slobodskoi proves  that if $k\in S_1$ then $\Psi(k)$ is true in every periodic group (in particular every finite group) \cite[Lemma 6]{slobodskoi_undecidability_1981}. He then proves that  if $k\in S_0$ then $\Psi(k)$ is false in some finite group \cite[Lemma 7]{slobodskoi_undecidability_1981}. 

\begin{proof}[Proof of Theorem \ref{thm: Slobodskoi}]
Let $G=\<A\mid R\>$ be the group constructed by Slobodskoi.
The set $\{w\in F(A)\mid w\neq_{\widehat{G}} 1\}$ is   recursively enumerable: a naive search will eventually find a finite quotient of $G$ in which $w$ survives, if one exists. If the complement $\{w\in F(A)\mid w=_{\widehat{G}} 1\}$ were recursively enumerable, then the set
\[
D= \{k\in\mathbb{N}\mid w_1^{(k)}=_{\widehat{G}}w_2^{(k)}=_{\widehat{G}}1\}
\]
would be recursive.  But $S_1\subseteq D$ 
and $S_0\subseteq \mathbb{N}\smallsetminus D$, 
so this would contradict the fact that $S_0$ and $S_1$ are recursively inseparable.
\end{proof}

\begin{remark}
Kharlampovich proved an analogue of Slobodskoi's theorem for the class of finite nilpotent groups \cite{kharlampovich_universal_1983}.
\end{remark}

\begin{remark}
It follows easily from Theorem \ref{thm: Slobodskoi} and the Hopfian property of finitely generated profinite groups that there does not exist an algorithm that, given two finite presentations, can determine if the profinite completions of the groups presented are isomorphic or not.  It is much harder to prove that the isomorphism problem remains unsolvable if one restricts to completions of finitely presented, residually finite groups \cite{bridson_isomorphism_2013}.
\end{remark}

\section{A strategy for proving Theorem \ref{t:tech}}\label{s:strategy}

In this section we lay out a strategy for proving our main technical result, Theorem \ref{t:tech}.
It is useful to think of Theorem \ref{t:tech} as a machine that, given a word in the seed group $G$,
produces a group $G_w$ so that the (non)triviality of $w\in \wh G$ is translated into the (non)triviality 
of the profinite completion $\wh G_w$.  Although the techniques required to prove this are quite different from the arguments
used to prove the corresponding result for discrete groups (which are straightforward from a modern perspective),
the broad outline of the proof in that setting will serve us well as a framework on which to hang various technical results. 
The notation established here will be used consistently in later sections.

\subsection{The discrete case}\label{ss:disc}
We fix a finitely presented group $G=\langle A\mid R\rangle$ and seek an algorithm that, given a word $w\in F(A)$, will produce a finitely presented group
$G^w$ so that $G^w\cong\{1\}$ if  $w=_G1$ and  $G\hookrightarrow G^w$
if $w\neq_G1$.
The first such algorithm was described by Adyan \cite{adyan_algorithmic_1955,adyan_unsolvability_1957} and Rabin \cite{rabin_recursive_1958}. There are many ways to vary the construction; cf.~\cite{miller_iii_decision_1992}.

\medskip

Replacing $G$ by $G*\langle a_0\rangle$ and  $a\in A$ by $a'= aa_0$, if necessary,
we may assume that $A=\{a_0,\dots,a_m\}$ where each $a_i$ has infinite order.
And replacing $w$ by $[w,a_0]$, we may assume that if $w$ is non-trivial in $G$ then it has infinite order.

Let $
G_1=G*\langle b_0,\ldots, b_m\rangle/\langle\!\langle w^{b_i}=a_i\mid i=0,\ldots, m\rangle\!\rangle
$ and let  $G_2=G_1*\langle b_{m+1}\rangle$.  Note that if $w=1$ then $G_2$ is freely generated by the $b_i$, whereas if $w\neq 1$ then $G_2$ is a multiple HNN extension of $G$ with stable letters $b_i$, whence the natural map $G\to G_2$ is injective.  Choose $m+2$ words that freely generate a 
subgroup of the normal closure of $w\in F(w,\,b_{m+1})$, say
$
c_j=(w^{b_{m+1}})^{j+1}w(w^{b_{m+1}})^{-1-j}.
$
Define $F_0:= \langle b_0,\ldots, b_m\rangle<G_1$, and further define subgroups of $G_2$ by
$$F_1:=\langle b_0,\ldots , b_{m+1}\rangle \ \ 
F_2:=\langle c_0,\ldots, c_{m+1}\rangle\ \ F:=\langle F_1,F_2 \rangle.
$$
The subgroup $F_1 < G_2$  is free of rank $m+2$.
If $w=_G 1$ then $F_2<G_2$ is trivial.
If $w\neq_G 1$ then  
 $F_2$ is free of rank $m+2$ and $F=F_1\ast F_2$ is the free product. 

We take two copies $G_2$ and $G_2'$ of $G_2$ and distinguish the elements and subgroups of $G_2'$ by primes.
Define $G^w$ to be the quotient of $G_2*G'_2$ by the relations
\[
\{c_i=b'_i, b_i=c'_i\mid i=0,\ldots,m+1\}~.
\]
If $w=_G1$ then $G^w\cong 1$.  If $w\neq_G 1$ then $G^w$ is an amalgamated product
\[
G_2*_{F\cong F'} G'_2
\]
where the isomorphism $F\cong F'$ identifies $F_1$ with $F_2'$ and $F_2$ with $F_1'$.
In particular, 
the natural map $G\to G_2\to G^w$ is injective and  $G^w\ncong 1$.

\subsection{The profinite case}
Given  $G=\<A\mid R\>$ and   $w\in F(A)$, we have to construct, in an algorithmic manner, a finite presentation for a group $G_w$ so that $\wh{G}_w=1$ if and only if $w=_{\wh{G}}1$. The difficult thing to arrange is that $G_w$ must have some non-trivial finite quotient if $w\neq_{\wh{G}}1$.

\begin{remark}
Of the many problems one faces in adapting the preceding argument to the profinite setting, the most fundamental concerns our use of HNN (equivalently, Bass--Serre) theory to see that the natural map $G\to G^w$  is injective if $w\neq 1$. Sobering examples in this connection are the {\em simple} groups of Burger and Mozes \cite{burger_finitely_1997}: these are amalgamations $L_1\ast_{\Lambda_1\cong \Lambda_2} L_2$ where $L_1\cong L_2$ is a finitely generated free group and $\Lambda_i<L_i$ is a subgroup of finite index. (Earlier examples in a similar vein were given by Bhattacharjee \cite{bhattacharjee_constructing_1994} and Wise \cite{wise_complete_2007}.)
\end{remark}

\subsection*{Step 1: controlling the order of the generators $a_i$ and of $w$}

What matters now is the order of $a_i$ and $w$ in finite quotients of $G$. In order to
retain enough finite quotients after performing the HNN extensions in step 2, we must ensure that 
if $w\neq 1$ in $\wh{G}$ then $w$ and the generators all have the same order in {\em some} finite quotient of $G$
(or a proxy of $G$).
It will transpire that in fact we need significantly  more control than
this. This control is established in Section \ref{s:omni}, where the key result is Theorem \ref{thm: Strengthened omnipotence}. 

\subsection*{Step 2: a map $G\to \wh{G}_1$ whose image is trivial iff $w=1$} 
We define $G_1$ as above, making the $a_i$ conjugate to $w$. If $w=_{\wh{G}}1$   then
$G\to \wh{G}_1$ is trivial and $\wh{G}_1$ is the profinite completion of the free group
$F_0$ on the stable letters $b_i$. When $w\neq_{\wh{G}}1$, we obtain finite quotients of $G_1$
in which $w$ survives. But this is not enough: for reasons that will become apparent in step 4, we have to work hard to find virtually free quotients $\eta:G_1\to\Gamma_0$ where $F_0$ injects and is {\em malnormal}.
 
We remind the reader that a subgroup $H<G$ is termed {\em malnormal} if $g^{-1}Hg\cap H=1$ for all $g\notin H$. This is the central concept of Section \ref{s: Malnormal subgroups} and continues to be a major focus in Section \ref{s:proof}.

\subsection*{Step 3: the construction of $\Gamma$ and $F$}
In $G_2=G_1\ast\<b_{m+1}\>$ we have
 to demand far more of the subgroup $F_2$ than in the discrete case. Consequently, a much more
subtle construction of the elements $c_i$ is required, and this is 
the subject of Section \ref{s: Malnormal subgroups}.  If $w=_{\wh{G}}1$ then
$F_2$ is trivial
in every finite quotient of $G_2$. If $w\neq_{\wh{G}}1$ 
then $F_1\cong F_2$ and $F\cong F_1\ast F_2$ injects into a virtually free quotient 
$\Gamma$ of $G_2$
(Lemma \ref{lem: Free product}) where it  is {\em malnormal} (Proposition \ref{prop: Malnormality of F}).

\subsection*{Step 4}

With our more sophisticated definition of $c_i$ and $F$ in hand, we  define
$G_{w}$ to be the quotient of $G_2*G'_2$ by the relations
\[
\{c_i=b'_i, b_i=c'_i\mid i=0,\ldots,m+1\}~.
\]
It is clear that $\wh{G}_w=1$ if $w=_{\wh{G}}1$. If $w\neq_{\wh{G}}1$, then 
$G_w$  maps onto $\G\ast_{{F}\cong{F}'}\G'$; as a malnormal amalgamation of virtually free groups, this is residually finite, by a theorem of Wise \cite{wise_residual_2002}.  

\medskip

\begin{remarks}

(1) A crucial feature of the above process is that each step is algorithmic: judicious choices
are made, but these choices depend in an algorithmic manner on the parameter $w$ alone.
In particular, the algorithm gives
an explicit finite presentation for $G_w$.

(2) The definition of $G_w$ makes no assumption about the existence or nature of the finite quotients of
$G$ in which $w$ has non-trivial image.
Equally, the proof that $G_w$ has a non-trivial finite quotient if $w\neq_{\wh{G}}1$
requires only the {\em existence} of a finite quotient in which $w$ has non-trivial image;
it does not require any knowledge about the nature of such a quotient.
\end{remarks}

\section{A strengthening of omnipotence}\label{s:omni}

The main result of this section (Theorem \ref{thm: Strengthened omnipotence})
strengthens Wise's theorem on the {\em omnipotence} of free
groups \cite{wise_subgroup_2000}.

Given a virtually free group $\G$ and a finite list of elements $\g_1,\dots, \g_n\in\G$, we would  like to control the (relative) orders of these elements in finite quotients of $\G$. Ideally, we would like to dictate orders arbitrarily, but this is too much to expect. For example, if $\g_1$ and $\g_2$ have conjugate powers in $\G$, then the possible orders for the image of $\g_2$ are constrained by those of $\g_1$. To isolate this problem, we make the following definition.

\begin{definition}\label{def: Independent}
Let $\G$ be a group. Elements $\g_1,\g_2\in \G$ of infinite order are said to be {\em{independent}} if  no non-zero power of $\g_1$ is conjugate to a non-zero power of $\g_2$. An $m$-tuple $(\gamma_1,\ldots,\gamma_m)$ of elements from $\G$ is  \emph{independent} if $\g_i$ and $\g_j$ are independent whenever $1\le i < j\le m$.
\end{definition}

The next definition 
makes precise the idea that the orders of independent sets of elements can be controlled in finite quotients. 

\begin{definition}\label{d:omnip}
A  group $\Gamma$ is \emph{omnipotent} if, for every $m\ge 2$ and every independent $m$-tuple $(\gamma_1,\ldots,\gamma_m)$ of elements in $\G$,   there exists a positive integer $\k$ such that, for every $m$-tuple of natural numbers $(e_1,\ldots,e_m)$ there is a homomorphism to a finite group
\[
q:\Gamma\to Q
\]
such that $o(q(\gamma_i))=\k e_i$ for $i=1,\dots,m$.
\end{definition}

The preceding definitions are due to Wise \cite{wise_subgroup_2000}, who proved that free groups are omnipotent.
Bajpai extended this to surface groups \cite{bajpai_omnipotence_2007}, and the second author proved that all Fuchsian groups are omnipotent \cite{wilton_virtual_2010}.  It follows from Wise's recent deep work on special cube complexes (specifically, from the Malnormal Special Quotient Theorem \cite{wise_structure_2012}), that virtually special groups are omnipotent.  In particular, virtually free groups are known to be omnipotent.  However, we do not want to obscure our current setting with the extra complications of special cube complexes and, more importantly, Wise's method of proof does not provide the additional strengthening contained in item (2) of the following theorem. This refinement is a vital component of the strategy described in the previous section:
it will be needed to establish malnormality in Lemma \ref{lem: Malnormality of F_0} and Proposition \ref{prop: Malnormality of F}.

\begin{theorem}\label{thm: Strengthened omnipotence}
Let $\Gamma$ be a virtually free group and let $(\gamma_1,\ldots,\gamma_m)$ be an independent $m$-tuple of elements of $\Gamma$.  There is a positive integer $\kappa$ such that, for every $m$-tuple of positive integers $(e_1,\ldots,e_m)$, there is a homomorphism to a finite group
\[
q:\Gamma\to Q
\]
such that:
\begin{enumerate}
\item $o(q(\gamma_i))=\k e_i$ for $i=1,\dots,m$; and,
\item furthermore, $\langle q(\gamma_i)\rangle\cap\langle q(\gamma_j)\rangle=1$ whenever $i\neq j$.
\end{enumerate}
\end{theorem}

The following lemma is a key step in the proof of omnipotence for free groups \cite{wise_subgroup_2000} (see also \cite[Proposition 4.1]{wilton_virtual_2010}).

\begin{lemma}\label{lem: Free omnipotence lemma}
Let $\Lambda$ be a finitely generated free group.  If $(\gamma_1,\ldots,\gamma_m)$ is an independent
$m$-tuple in $\Lambda$, then there exists a 
subgroup $F<\Lambda$ of finite index and 
homomorphisms  $\phi_i: F\to\mathbb{Z}$ such that the restriction of $\phi_i$ to $F\cap \<\gamma_i\>$ is non-trivial  but
$\phi_i(f)=0$ if $f\in F\cap\<\delta\gamma_j\delta^{-1}\>$ for any $j\neq i$ and $\delta\in \Lambda$. 
\end{lemma}
\begin{proof}  
We identify $\Lambda$ with the fundamental group of a finite connected graph $Y$.  By Marshall Hall's theorem, for each $i$ there exists a finite-sheeted covering space $Y_i\to Y$ in which $\gamma_i$ is represented by an embedded loop.  Let $X\to Y$ be a regular, finite-sheeted covering space of $Y$ that factors through $Y_i$ for every $i$. Note that
 the generator of $\pi_1X\cap\<\delta\gamma_j\delta^{-1}\>$
is represented by an embedded loop in $X$, say $\lambda_{j,\delta}$, for all $j=1,\dots,m$ and $\delta\in\Lambda$. In 
\cite[Theorem 3.6]{wise_subgroup_2000}, Wise proves that given any graph $G$ and any simple loop $\lambda$
in that graph, there is a finite-sheeted covering $\check{G}\to G$ in which any elevation of $\lambda$
(i.e. a lift of a power of $\lambda$) is independent in $H_1(\check{G},\Z)$ from the full set of elevations of all other simple loops in $G$.
Applying this to the loop $\lambda_{i,1}$ in $X$, we obtain a finite-sheeted covering $X_i\to X$ 
and a homomorphism $\psi_i$ from $F_i:=\pi_1X_i$ to $\mathbb{Z}$ such that $\psi_i|_{F_i\cap\langle\gamma_i\rangle}$ is non-trivial but $\psi_i(F_i\cap\<\delta\gamma_j\delta^{-1}\>)=0$ for all $\delta\in \Lambda$ if $j\neq i$.  Taking $F$ to be the intersection of the $F_i$ and $\phi_i=\psi_i|_F$ completes the proof.
\end{proof}

We need to improve Lemma \ref{lem: Free omnipotence lemma} to deal with virtually free groups $\Gamma$.

\begin{lemma}\label{lem: Main omnipotence lemma}
Let $\Gamma$ be a virtually free group.  If $(\gamma_1,\ldots,\gamma_m)$ is 
an independent $m$-tuple, then there exists a free, normal subgroup $F<\Gamma$ of finite index
and homomorphisms  $\phi_i: F\to \mathbb{Z}$ such that the restriction of $\phi_i$ to $F\cap \<\gamma_i\>$ is non-trivial  but
$\phi_i(f)=0$ if $f\in F\cap\<\delta\gamma_j\delta^{-1}\>$ for any $j\neq i$ and $\delta\in \Gamma$. 
\end{lemma}

\begin{proof}
By hypothesis, there is a short exact sequence of groups
\[
1\to \Lambda \to \G \to \Sigma \to 1
\]
with $\Lambda$ free and $\Sigma$ finite. 

Given independent $\g_1,\dots,\g_m\in\G$, we may replace the $\gamma_i$ by proper powers to assume that each $\gamma_i\in\Lambda$.  
Then, we enlarge our list of elements by adding to it elements of  $\Lambda$
 that are conjugate to some $\g_i$ in $\G$ but not in $\Lambda$. To this end, we fix a set of coset representatives $\widetilde{\Sigma} = \{\tilde \sigma \mid \sigma\in \Sigma\}$   for $\Lambda$ in $\G$, with $\tilde 1 =1$, and define $ g_{i\sigma}= \tilde \sigma \g_i \tilde \sigma^{-1}$.    Since the $\g_i$ are independent, no element of $\{ g_{i\sigma}\mid \sigma\in \Sigma\}$ has a non-zero power that is conjugate to a non-zero power of an element of $\{ g_{j\sigma}\mid \sigma\in \Sigma\}$ if $i\neq j$. However, the indexed set  $( g_{i\sigma}\mid \sigma\in \Sigma)$ may fail to be independent since it is quite possible that $ g_{i\sigma}$ will be conjugate to $  g_{i\sigma'}^{\pm 1}$ for some $\sigma\neq \sigma'$.   (In a virtually free group an element of infinite order $x$ cannot be conjugate to $x^p$  with $|p|>1$, so higher powers are not a worry.) To account for such coincidences we make deletions from the list $( g_{i\sigma}\mid \sigma\in \Sigma)$, reducing it to $( g_{i\sigma}\mid \sigma\in \Sigma[i])$, say.  This reduced list consists of a set of orbit representatives for the action of $\G$ by conjugation on the $\Lambda$-conjugacy classes of cyclic subgroups of the form $\<f \g_i f^{-1}\>$ with $f\in \Gamma$. 

We now apply Lemma \ref{lem: Free omnipotence lemma} to the concatenation of the lists $( g_{1\sigma}\mid \sigma\in \Sigma[1]),\dots, ( g_{m\sigma}\mid \sigma\in \Sigma[m])$, which is independent in $\Lambda$. Thus we obtain  a free subgroup of finite index $F<\Lambda$ and homomorphisms
$\phi_{i\sigma}: F\to \mathbb{Z}$ with the property that $\phi_{i1}$ is non-trivial on $F\cap\<\g_i\>$ but
$\phi_{i1}(f)=0$ if $f\in F\cap\<\delta g_{j\sigma} \delta^{-1}\>$ for any  $\delta\in \Lambda$ and $(j,\sigma)\neq (i,1)$. Moreover, since these conditions
are inherited by subgroups of finite index in $F$, we may replace $F$ by a smaller subgroup if necessary to ensure that it is normal in $\Gamma$.

Henceforth we write $\phi_i$ in place of $\phi_{i1}$.

Consider $\delta\in\Gamma$ and $ \gamma_j$ with $j\neq i$.  Write $\delta=\delta'\tilde{\sigma}$ for some $\delta'\in \Lambda$.  If $\tilde{\sigma}\in \Sigma[j]$ then 
for any positive power $n$ such that $\gamma_{j}^n\in F$ we have $\phi_i(\delta \gamma_j^n\delta^{-1})=\phi_i(\delta' g_{j\sigma}^n(\delta')^{-1})=0$ as required.  On the other hand, if $\tilde{\sigma}\notin \Sigma[j]$ then there exists $\lambda\in \Lambda$ such that $ g_{j\sigma}=\lambda g_{j\sigma'}^{\pm 1}\lambda^{-1}$ for some $\sigma'\in \Sigma[j]$.  Then,
\[
\phi_i(\delta \gamma_{j}^n\delta^{-1})=\phi_i(\delta'\lambda g_{j\sigma'}^{\pm n}\Lambda^{-1}(\delta')^{-1})=0~,
\]
which finishes the proof.
\end{proof}

With Lemma \ref{lem: Main omnipotence lemma} in hand, we can prove Theorem \ref{thm: Strengthened omnipotence}.

\begin{proof}[Proof of Theorem \ref{thm: Strengthened omnipotence}]
Let $\Gamma$, $F$ and $\phi_i:F\to\Z$ be as in Lemma \ref{lem: Main omnipotence lemma} and let $\eta:\Gamma\to\Gamma/F$ be the quotient map.  Let $\ell_i=\phi_i(\gamma_i^{o(\eta(\gamma_i))})$ and note that there is no loss of generality in assuming that $\ell_i$ is positive.
 Fix a set of coset representatives $c_j$ for $F$ in $\Gamma$ with $c_1=1$. For each $\gamma_i$, fix a positive integer $N_i$ (to be specified later) and consider the composition
\[
\psi_i:F\stackrel{\phi_i}{\to} \mathbb{Z}\to\mathbb{Z}/N_i~.
\]
Then, consider the direct product
\[
\Psi_i=\prod_j \psi_i\circ i_{c_j}:F\to A_i=\prod_j \mathbb{Z}/N_i
\]
where $i_{c_j}$ is the automorphism of $F$ given by conjugation by $c_j$. It is now clear that
$o(\Psi_i(\gamma_i^{o(\eta(\gamma_i))}))= N_i/\ell_i$, whereas
\begin{equation*}
\Psi_i(\gamma_k^{o(\eta(\gamma_k))})= 0
\end{equation*}
for all $k\neq i$.  The direct product
\[
\Psi=\prod_i\Psi_i:F\to A=\prod_i A_i
\]
therefore has the property that
\[
o(\Psi(\gamma_i)^{o(\eta(\gamma_i))})= N_i/\ell_i
\]
for all $i$.  Now, $\Psi$ is the restriction to $F$ of the homomorphism
\[
\Phi:\Gamma\to A\rtimes (\G/F) = \left(\prod_{i}\Z/N_i\right)\wr (\G/F)  
\]
induced from $\prod_i\psi_i:F\to\prod_i\Z/N_i$.  Therefore, $o(\Phi(\gamma_i))=N_io(\eta(\gamma_i))/\ell_i$ for all $i$. 

To prove the theorem, we define $Q$ to be $A\rtimes (\G/F)$ and $q$ to be $\Phi$, then we take
$\kappa=|\Gamma/F|^2$ and $N_i=\ell_i\kappa e_i/o(\eta(\gamma_i))$. The preceding computation shows that
$o(\Phi(\gamma_i))=\kappa e_i$, which proves the first assertion.

To prove the second assertion, suppose that an intersection,
\[
\langle\Phi(\gamma_1)\rangle\cap\langle\Phi(\gamma_2)\rangle
\]
say, is non-trivial.  Then it contains a minimal non-trivial subgroup, which is of prime order $p$. That is, the intersection contains the non-trivial subgroup
\[
\langle \Phi(\gamma_1^{\kappa e_1/p})\rangle=\langle \Phi(\gamma_2^{\kappa e_2/p})\rangle~.
\]
We have
\[
o(\eta(\gamma_i))~|~\kappa e_i/p
\]
(because $\kappa=|\Gamma/F|^2$), and so $\gamma_i^{\kappa e_i/p}\in F$, for $i=1,2$.   Therefore $\Psi(\gamma_i^{\kappa e_i/p})=\Phi(\gamma_i^{\kappa e_i/p})$ for $i=1,2$, and so 
\[
\langle \Psi(\gamma_1^{\kappa e_1/p})\rangle=\langle \Psi(\gamma_2^{\kappa e_2/p})\rangle~.
\]
One of the coordinates of the homomorphism $\Psi$ is $\psi_1$, and so it follows that
\begin{equation*}
\langle \psi_1(\gamma_1^{\kappa e_1/p})\rangle=\langle \psi_1(\gamma_2^{\kappa e_2/p})\rangle~.
\end{equation*}
But this leads to a contradiction because,
on the one hand, we have $\psi_1(\gamma_2^{\kappa e_2/p})=0$ by the definition of $\psi_1$ and Lemma \ref{lem: Main omnipotence lemma},
while on the other hand, $\psi_1(\gamma_1^{\kappa e_1/p})\neq 0$, because
\[
\psi_1(\gamma_1^{\kappa e_1/p})=\psi_1(\gamma_1^{o(\eta(\gamma_1))})^{\kappa e_1/po(\eta(\gamma_1))}
\]
and $\kappa e_1/po(\eta(\gamma_1))$ is less than $o(\psi_1(\gamma_1^{o(\eta(\gamma_1))}))=(N_1/\ell_1)=\kappa e_1/o(\eta(\gamma_1))$.
\end{proof}

\section{Constructing Malnormal Subgroups}\label{s: Malnormal subgroups}

The role that malnormality plays in our strategy was explained in Section \ref{s:strategy}.   The main result in this section is Proposition \ref{prop: m+2-generator malnormal subgroup}, but several of the other lemmas will also be required in the next section.  Fibre products of morphisms of graphs, as described by Stallings \cite{stallings_topology_1983}, play a prominent role in many of our proofs.

\begin{definition}
Let $\Gamma$ be a group and $H$ a subgroup.  Then $H$ is said to be \emph{almost malnormal} in $\Gamma$ if $|H\cap H^\gamma|<\infty$ whenever $\gamma\in\Gamma\smallsetminus H$.  If we in fact have $H\cap H^\gamma=1$ whenever $\gamma\in\Gamma\smallsetminus H$ then $H$ is said to be \emph{malnormal}.

More generally, a family $\{H_i\}$ of subgroups of $\Gamma$ is said to be \emph{almost malnormal} if $|H_i\cap H_j^\gamma|=\infty$ implies that  $i= j$ and $\gamma\in H_j$.  Similarly, we may speak of malnormal families of subgroups.
\end{definition}

Note that if $H$ is torsion-free and almost malnormal then it is in fact malnormal.

The first fact we record is trivial but extremely useful.

\begin{lemma}\label{lem: Transitive malnormal}
If $K$ is an (almost) malnormal subgroup of $H$ and $H$ is an almost malnormal subgroup of $G$ then $K$ is an (almost) malnormal subgroup of $G$.
\end{lemma}

The next lemma, which again admits a trivial proof, enables one to deduce almost malnormality from virtual considerations.

\begin{lemma}\label{lem: Virtual malnormal}
Let $H$ be an arbitrary subgroup of a group $\Gamma$ and let $\Gamma_0$ be a subgroup of finite index in $\Gamma$.  Fix a set of double-coset representatives $\{\gamma_i\}$ for $H\backslash\Gamma/\Gamma_0$. Then $H$ is almost malnormal in $\Gamma$ if and only if the family $\{H^{\gamma_i}\cap\Gamma_0\}$ is almost malnormal in $\Gamma_0$.
\end{lemma}

The malnormality of a family of subgroups of a free group can be determined by a computation using the elegant formalism of fibre products, as we will now explain.

Consider a pair of immersions of finite graphs $\iota_1:Y_1\to X$ and $\iota_2:Y_2\to X$.  Recall that the \emph{fibre product} of the maps $\iota_1$ and $\iota_2$ is defined to be the graph
\[
Y_1\times_X Y_2=\{(y_1,y_2)\in Y_1\times Y_2\mid \iota_1(y_1)=\iota_2(y_2)\}~.
\]
The fibre product comes equipped with a natural immersion $\kappa:Y_1\times_X Y_2\to X$.  For any $(y_1,y_2)$, Stallings pointed out that
\[
\kappa_*\pi_1(Y_1\times_X Y_2,(y_1,y_2))=\iota_{1*}\pi_1(Y_1,y_1)\cap\iota_{2*}\pi_1(Y_2,y_2)
\]
\cite[Theorem 5.5]{stallings_topology_1983}.  In the case when $Y_1=Y_2=Y$ and $\iota_1=\iota_2$, there is a canonical diagonal component of $Y\times_XY$, isometric to $Y$.

The next lemma follows immediately from this discussion.

\begin{lemma}\label{lem: Fibre product}
Let $X$ be a connected finite graph with fundamental group $F$, and let $Y$ be a (not necessarily connected) finite graph equipped with an immersion $Y\to X$.  The components $\{Y_i\}$ of $Y$ define (up to conjugacy) a family of subgroups $H_i$ of $F$.  Then $\{H_i\}$ is malnormal if and only if every non-diagonal component of the fibre product $Y\times_X Y$ is simply connected.
\end{lemma}

In particular, this gives an algorithm to determine whether or not a given family of subgroups of a free group is malnormal.

Unlike Lemma \ref{lem: Fibre product}, the next lemma is not always applicable.  However, it gives a useful sufficient condition for malnormality, which can sometimes be applied in situations where Lemma \ref{lem: Fibre product} is too cumbersome to apply in practice.  Let $Z_\Gamma(g)$ denote the centralizer of an element $g$ in a group $\Gamma$.

\begin{lemma}\label{lem: Malnormal retract}
Let $H$ be a subgroup of $\Gamma$.  If $H$ is a retract and $Z_\Gamma(h)\subseteq H$ for all $h\in H\smallsetminus 1$, then $H$ is malnormal.
\end{lemma}
\begin{proof}
Let $\rho:\Gamma\to H$ be a retraction.  Suppose that $h\in H\smallsetminus 1$ and $h^\gamma\in H$. Then $h^\gamma=h^{\rho(\gamma)}$, which implies that $\gamma\rho(\gamma)^{-1}\in Z_\Gamma(h)$ and so $\gamma\in H$, as required.
\end{proof}

We now develop some simple examples.

\begin{example}
If $\Gamma$ is a group and $H$ is a free factor then $H$ is malnormal in $\Gamma$.  This is an immediate consequence of Lemma \ref{lem: Malnormal retract}, since free factors are retracts.
\end{example}

The following easy example, which will be useful later, illustrates how Lemmas \ref{lem: Virtual malnormal} and \ref{lem: Fibre product} can be used to prove almost malnormality in virtually free groups.

\begin{example}\label{ex: Free product malnormal subgroup}
Suppose that $A$ is a finite group and $B$ is any subgroup of $A$.  Then the natural copy of $H=B*\mathbb{Z}$ inside $\Gamma=A*\mathbb{Z}$ is almost malnormal.

To see this, realize $\Gamma$ as the fundamental group of a graph of groups $\mathcal{X}$ with a single vertex labelled $A$ and a single edge with trivial edge group.  The kernel of the retraction $\Gamma\to A$ implicit in the notation is a normal, free subgroup $F$ of finite index.  Let $T$ be the Bass--Serre tree of $\mathcal{X}$.  The quotient $F\backslash T$ is a graph $X$ with a single vertex and $|A|$ edges $\{e_a\mid a\in A\}$, and the natural $A$-action is by left translation.  The subgroup $H\cap F$ is carried by the subgraph $Y=\bigcup_{b\in B}e_b$.

The quotient map $\Gamma\to A$ identifies $H\backslash\Gamma/F$ with $B\backslash A$, so a set of double-coset representatives for the former is provided by any set $\{a_i\}$ of right-coset representatives for $B$ in $A$. The subgroup $H^{a_i}\cap F$ is carried by the subgraph $a_i^{-1}Y$: under the immersion
\[
Z=\coprod_i a_i^{-1}Y\to X
\]
(where the map $Z\to X$ is inclusion on each component), the fundamental groups of the components are mapped to the family of subgroups $\{H^{a_i}\cap F\}$.

Note that, as subgraphs of $X$, $a^{-1}_iY$ and $a^{-1}_jY$ have no edges in common if $i\neq j$.  Therefore, every off-diagonal component of $Z\times_X Z$ is a vertex and hence simply connected. 

It follows that $\{H^{a_i}\cap F\}$ forms a malnormal family in $F$ by Lemma \ref{lem: Fibre product}, and so $H$ is almost malnormal in $\Gamma$ by Lemma \ref{lem: Virtual malnormal}. 
\end{example}

The following construction provides us with the supply of malnormal subgroups that we shall need to prove Theorem \ref{thm: Main theorem}.

\begin{lemma}\label{lem: 3-generator malnormal subgroup}
Let $\Lambda_2\cong \langle \alpha,\beta\rangle$ be free of rank two.  For each integer $N$,  let
\[
q_N:\Lambda_2\to Q_N= \Lambda_2/\langle\!\langle \beta^N\rangle\!\rangle
\]
be the quotient map.  Consider $u=\alpha^\beta(\alpha^{\beta^2})^{-1}$, $v=\alpha^{\beta}(\alpha^{\beta^{-1}})^{-1}$.  For all $N>6$, the subgroup $q_N(\langle \alpha,u,v\rangle)$ is malnormal in $Q_N$ and free of rank 3.
\end{lemma}
\begin{proof}
Consider the images $\bar{\alpha}=q_N(\alpha)$, $\bar{\beta}=q_N(\beta)$, $\bar{u}=q_N(u)$ and $\bar{v}=q_N(v)$.  Let $F$ be the kernel of the retraction $Q_N\to \Z/N$ that maps $\bar{\alpha}\mapsto 0$ and $\bar{\beta}\mapsto 1$.  As above, $F$ may be thought of as the fundamental group  of a graph $X$ with a single vertex, and with $N$ edges $\{e_i\}_{i\in\Z/N}$, on which $\mathbb{Z}/N=\langle \bar{\beta}\rangle$ acts by left translation.   Represent $\langle \bar{\alpha},\bar{u},\bar{v}\rangle$ by the usual immersion of core graphs $\iota:Y\to X$.  As long as $N\geq 4$, the core graph $Y$ is easily computed explicitly using Stallings folds (see Figure \ref{fig: Core Y}), and is seen to have rank 3 as required.

\begin{figure}[htp]
\begin{center}
 \centering \def\svgwidth{75pt}
 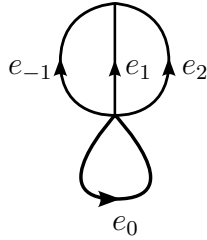
\caption{The core graph $Y$.  The edges are labelled with their images in $X$.}\label{fig: Core Y}
\end{center}
\end{figure}

Let $Y_i=Y$ for each $i=0,\ldots,N-1$, and consider the disjoint union
\[
Z=\coprod_{i=0}^{N-1} Y_i\to X
\]
where the map on $Y_i$ is $\bar{\beta}^i\circ\iota$.  To prove malnormality, it suffices to argue that every off-diagonal component of the fibre product $Z\times_XZ$ is simply connected.

Suppose some off-diagonal component is not simply connected. Translating by an element of $\langle\bar{\beta}\rangle$, we may assume that it arises as part of the fibre product $Y_0\times_X Y_i$ for some $i$. Since the image of $\iota$ only contains the edges $e_{i}$ for $-1\leq i\leq 2$, this fibre product contains no edges unless $0\leq i\leq 3$ (because $N>6$).

Therefore, it is enough to check that the off-diagonal components of $Y_0\times_X Y_0$ are simply connected, and that every component of $Y_0\times_X Y_i$ is simply connected, where $i=1,2,3$.  The off-diagonal components of $Y_0\times_X Y_0$ are points; for $i=1,2,3$, the fibre product $Y_0\times_X Y_i$ has $4-i$ edges, and a direct computation shows that each of these is a forest.  The fibre products $Y_0\times_X Y_0$ and $Y_0\times_X Y_1$ are illustrated in Figure \ref{fig: Fibre products}, while $Y_0\times_X Y_2$ and $Y_0\times_XY_3$ are left as easy computations for the reader. 
\end{proof} 

\begin{figure}[htp]
\begin{center}
 \centering \def\svgwidth{350pt}
 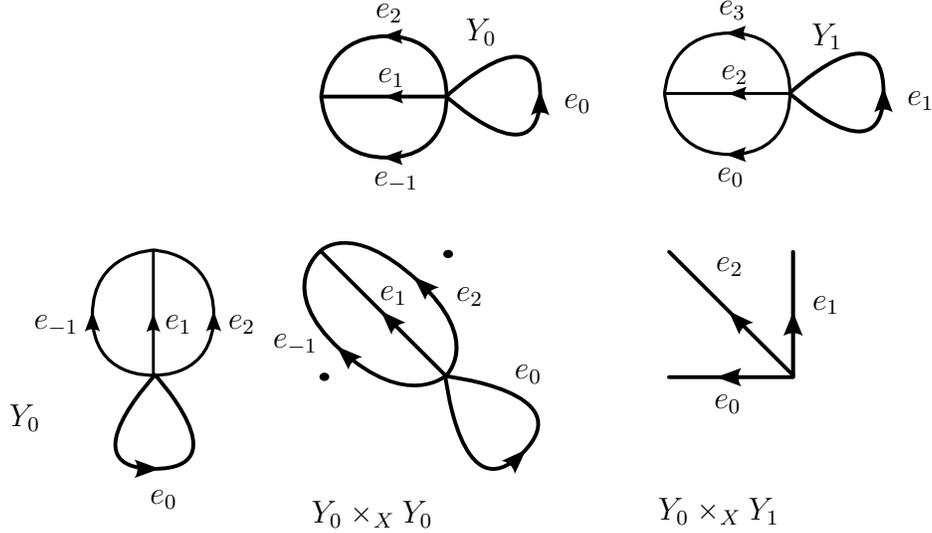
\caption{The fibre products $Y_0\times_X Y_0$ and $Y_0\times_X Y_1$, displayed as subsets of the direct products $Y_0\times Y_0$ and $Y_0\times Y_1$.  Note that the only non-simply-connected component is the diagonal component of $Y_0\times_X Y_0$.}\label{fig: Fibre products}
\end{center}
\end{figure}

From the 3-generator case, we immediately obtain malnormal subgroups with arbitrarily many generators.

\begin{proposition}\label{prop: m+2-generator malnormal subgroup}
Let $\Lambda_2\cong \langle \alpha,\beta\rangle$ be free of rank two.  For each integer $N$,  let
\[
q_N:\Lambda_2\to Q_N= \Lambda_2/\langle\!\langle \beta^N\rangle\!\rangle
\]
be the quotient map. For any $m$, there exist $\{\gamma_0,\ldots,\gamma_{m+1}\}\in [\Lambda_2,\Lambda_2]$ such that, for all $N>6$, the subgroup $q_N(\langle \alpha,\gamma_0,\ldots,\gamma_{m+1}\rangle)$ is malnormal in $Q_N$ and free of rank $m+3$.
\end{proposition}
\begin{proof} Let $L=\langle \gamma_0,\ldots,\gamma_{m+1}\rangle$ 
be any rank-$(m+2)$ malnormal subgroup of the free group $\langle u,v \rangle$  
constructed in Lemma \ref{lem: 3-generator malnormal subgroup}. Then $L\ast\<\alpha\>$ is malnormal in 
$\<\alpha,u,v\>$ and hence in $Q_N$, by Lemma \ref{lem: Transitive malnormal}. Also, since $u$ and $v$ lie in $[\Lambda_2,\Lambda_2]$, the $\gamma_i$ 
do as well.
\end{proof}

\section{The proof of Theorem \ref{thm: Profinite groups from words}}\label{s:proof}

In this section we prove Theorem \ref{thm: Profinite groups from words}, following the
strategy laid out in Section \ref{s:strategy}.
  As mentioned in the introduction, Theorem \ref{thm: Main theorem} follows immediately, using Theorem \ref{thm: Slobodskoi}.

We are given
 a finitely presented group $G=\langle A\mid R\rangle=\langle a_1,\ldots, a_m\mid r_1,\ldots, r_n\rangle$ and a word $w\in F(A)$.

\subsection*{Step 1: improving the input}

We start by proving some lemmas that improve the input $G$ and $w$.

\begin{lemma}\label{lem: Injective generators}
There is an algorithm that takes as input a finitely presented group $G\cong\<A\mid R\>$ and a word $w\in F(A)$ and outputs a finite presentation $\langle A^\dagger\mid R^\dagger \rangle$ for a group $G^\dagger $ and a word $w^\dagger \in F(A^\dagger )$ such that:
\begin{enumerate}
\item $w^\dagger =_{\widehat{G}^\dagger }1$ if and only if $w=_{\widehat{G}}1$;
\item if $w^\dagger \neq_{\widehat{G}^\dagger } 1$ then the natural map $\{1\}\sqcup A^\dagger \to\widehat{G}^\dagger $ is an embedding.
\end{enumerate}
\end{lemma}
\begin{proof}
Take $2m+1$ copies $G^{(j)}$ of $G$ and let $a_{ij}$ be the copy of $a_i$ in $G^{(j)}$; similarly, let $w_j$ be the copy of $w$ in $G^{(j)}$.  We will always take the $j$ index modulo $2m+1$.  Set
\[
G^\dagger =G^{(1)}*\ldots *G^{(2m+1)}
\]
and note that $w_j=_{\widehat{G}^\dagger } 1$ if and only if $w=_{\widehat{G}} 1$.  Now consider the following generating set $A^\dagger$ for $G^\dagger $:
\[
\{a_{ij}w_{j+m+1}w_{i+j}\mid 1\leq i\leq m,~1\leq j\leq 2m+1\}\cup\{w_j\mid 1\leq j\leq 2m+1\}~.
\]
Let $G\to \Sigma$ be a finite quotient in which $w$ survives,  let $\eta_j:G^{(j)}\to \Sigma^{(j)}$ be the corresponding quotient of $G^{(j)}$
and consider the free product of the maps $\eta_j$
\[
\eta^\dagger :G^\dagger \to \Sigma^\dagger =\Sigma^{(1)}*\ldots*\Sigma^{(2m+1)}.
\]
Suppose now that
\[
\eta^\dagger (a_{ij}w_{j+m+1}w_{i+j})=\eta^\dagger (a_{i'j'}w_{j' +m+1}w_{i' +j' })
\]
for some $i,j,i' ,j' $.  Because $i,i' <m+1$, the three syllables of the product $a_{ij}w_{j+m+1}w_{i+j}$  lie in different free factors, hence $j=j' $ and $i=i' $.   Similarly, the images of the generators $w_j$  lie in unique and distinct free factors.  Therefore, the restriction of $\eta^\dagger$ to $1\sqcup A^\dagger$ is injective.  Since $\Sigma^\dagger$ is virtually free and hence residually finite, it follows that $1\sqcup A^\dagger$ injects into $\widehat{G}^\dagger$ as required.   Setting $w^\dagger =w_1$ finishes the proof.
\end{proof}


\begin{proposition}\label{prop: Controlled orders}
There is an algorithm that takes as input a finitely presented group $G\cong\<A\mid R\>$ and a word $w\in F(A)$ and outputs a finite presentation $\langle A'\mid R'\rangle$ for a group $G'$ and a word $w'\in F(A')$ such that:
\begin{enumerate}
\item $w'=_{\widehat{G}'}1$ if and only if $w=_{\widehat{G}}1$;
\item if $w\neq_{\widehat{G}}1$ then, for any $N\in\N$, there exists a homomorphism to a finite group $\eta:G'\to Q$ such that:
\begin{enumerate}
\item $o(\eta(a'))=o(\eta(w'))\geq N$ for all $a'\in A'$; and
\item $\langle \eta(a'_i)\rangle\cap \langle \eta(a'_j)\rangle=\langle \eta(a'_i)\rangle\cap \langle \eta(w')\rangle=1$ whenever $i\neq j$.
\end{enumerate}
\end{enumerate}
\end{proposition}
\begin{proof}
We may algorithmically construct a presentation $\langle A^\dagger\mid R^\dagger\rangle$ and a word $w^\dagger$ as in Lemma \ref{lem: Injective generators}.  Write $A^\dagger=\{a^\dagger_1,\ldots,a^\dagger_m\}$.  Let $G'=G^\dagger*\langle a'_0\rangle$, let $a'_i=a^\dagger_ia'_0$ for each $i$ and let $w'=[w^\dagger,a'_0]$.  
Assertion (1) is now immediate. 

Let $\eta^\dagger : G^\dagger\to\Sigma^\dagger$ be as in the proof of Lemma \ref{lem: Injective generators} and let $\G = \Sigma^\dagger \ast \< a_0\>$.
We extend $\eta$ to a surjection $\zeta: G'\to \G$ by defining $\zeta(a_0)=a_0$. The
map $\eta$ is injective on  $1 \sqcup A^\dagger$, so by the normal form theorem for free products,  
$(\zeta(a'_0),\ldots,\zeta(a'_m),\zeta(w'))$ is an independent $(m+2)$-tuple in $\G$. To complete the proof, we define $\eta$ to be the composition
of $\zeta$ and the map $q:\G\to Q$ provided by Theorem \ref{thm: Strengthened omnipotence}.
\end{proof}

To avoid being overwhelmed by notation, we rename $G'$ as $G$, $A'=\{a'_0,\ldots, a'_m\}$ as $A$ and $w'$ as $w$.

\subsection*{Step 2: a map $G\to \wh{G}_1$ whose image is trivial iff $w=1$} 

We define a new finitely presented group 
\[
G_1=G*\langle b_0,\ldots, b_m\rangle/\langle\!\langle w^{b_i}=a_i\mid i=0,\ldots, m\rangle\!\rangle
\]
and let $F_0$ denote the subgroup $\langle b_0,\ldots, b_m\rangle$.  Note that there is a retraction $\rho:G_1\to F_0$, whence $F_0$ is free of rank $m+1$.  Note too that there is
a simple algorithm for deriving a finite presentation of $G_1$ from  $G$ and $w$. The following
lemma is clear.

\begin{lemma}
If $w=_{\widehat{G}}1$, then the inclusion map $F_0\hookrightarrow G_1$ and the retraction $\rho$ induce isomorphisms of profinite completions. 
\end{lemma} 

 If $w\neq_{\widehat{G}}1$ then we have the finite quotient $\eta:G\to Q$ guaranteed by Proposition \ref{prop: Controlled orders}.  We extend $\eta$ to an epimorphism from $G_1$ to the virtually free group $\Gamma_0$ given by the relative presentation  below.  We will continue to denote this epimorphism by $\eta$ and, to further simplify notation, we will use bars to denote the image of an element or a subgroup under $\eta$, so $\eta(w)=\bar{w}$, $\eta(F_0)=\overline{F}_0$ etc.
$$
\G_0 = ( Q, \bar{b}_1, \dots, \bar{b}_m \mid \bar{w}^{\bar{b}_i}=\bar{a}_i \text{ for } i=0,\dots, m)
$$
Note that $o(\bar{w})=o(\bar{a}_i)$ in $Q$, by Proposition  \ref{prop: Controlled orders}, and therefore $\G_0$ is a multiple HNN extension of $Q$.  Let  $\mathcal{X}_0$ be the corresponding graph of groups and let $T_0$ be its Bass--Serre tree.

\begin{lemma}\label{lem: Malnormality of F_0}
If $w\neq_{\widehat{G}}1$ then, for all natural numbers $N$, the group $G_1$ has a virtually free quotient $\eta:G_1\to\Gamma_0$ with the following properties:
\begin{enumerate}
\item for all $a\in A$, $N\leq o(\bar{w})=o(\bar{a})<\infty$;
\item $\overline{F}_0$ is free of rank $m+1$ and malnormal in $\Gamma_0$.
\end{enumerate}
\end{lemma}
\begin{proof}
The map $\eta:G_1\to\Gamma_0$ was constructed above. In the light of Proposition 6.2, the only point that is not immediate is that $\overline{F}_0$ is malnormal.  The quotient of $\Gamma_0$ by $Q$ defines a retraction $\bar{\rho}$ from $\Gamma_0$ to $\overline{F}_0$. By Lemma \ref{lem: Malnormal retract}, it suffices to prove that $Z_{\Gamma_0}(h)\subseteq\overline{F}_0$ for all $h\in\overline{F}_0\smallsetminus 1$.

Suppose therefore that $h\in\overline{F}_0\smallsetminus 1$ and $[h,\gamma]=1$.  Let $h=\bar{b}_{i_1}^{\epsilon_1}\ldots \bar{b}_{i_k}^{\epsilon_k}$, where $\epsilon_i\in\{\pm1\}$ for all $i$.  We may assume that this decomposition is cyclically reduced, and therefore the vertex $*$ in $T_0$ stabilized by $Q$ is on the axis of $h$.

We claim that $\bar{\rho}(\gamma^{-1})\gamma\in Q$. Because $[h,\gamma]=1$, the segment $[*,\gamma*]$ is contained in $\mathrm{Axis}(h)$.  Because $\overline{F}_0$ acts on its minimal invariant subtree with a single orbit of vertices, there exists $\beta\in \overline{F}_0$ such that $\beta*=\gamma*$, and so $\beta^{-1}\gamma\in Q$.  Therefore
\[
1=\bar{\rho}(\beta^{-1}\gamma)=\beta^{-1}\bar{\rho}(\gamma)
\]
and the claim follows.

Since $\bar{\rho}(\gamma)\in Z_{\Gamma_0}(h)$, the claim reduces us to the case that $\gamma\in Q$, in which case $\gamma$ fixes the whole of $\mathrm{Axis}(h)$. But, by item (2)(b) of Proposition \ref{prop: Controlled orders}, no non-trivial element of $\Gamma_0$ fixes a subset of diameter greater than 2 in the minimal $\overline{F}_0$-invariant subtree of $T_0$, and therefore $\gamma=1$. 
\end{proof}

\subsection*{Step 3: the free subgroups $F_1,\, F_2$ and $F$} 

Let $G_2=G_1*\langle t\rangle$ and let $F_1$ be the free subgroup $F_0*\langle t\rangle$ of rank $m+2$.  It will later be convenient to write $b_{m+1}=t$. 
Casting $t$ and $w$ in the roles of $\alpha$ and $\beta$, 
we choose $c_j=\gamma_j$ as in Proposition \ref{prop: m+2-generator malnormal subgroup}, for $j=0,\ldots,m+1$, and write $F_2$ for the subgroup of $G_2$ generated by the $c_j$.

Since $c_j$ is in the commutator subgroup of $\langle t,w\rangle$, we have
$c_j\in \langle\!\langle w\rangle\!\rangle$ and hence the image of $F_2$ in $\widehat{G}_2$ is trivial
 if $w=_{\widehat{G}} 1$.

We analyse what happens
when $w\neq_{\widehat{G}} 1$.  Let $\eta:G_1\to \Gamma_0$ be the virtually free quotient guaranteed by Lemma \ref{lem: Malnormality of F_0}. We will extend $\eta$ to a homomorphism from $G_2$ onto a virtually free group $\Gamma$; we will then continue to denote this homomorphism by $\eta$, and continue to denote $\eta$-images by bars.

Consider the graph of groups $\mathcal{X}$ obtained from $\mathcal{X}_0$ by adjoining a single loop $e_{m+1}$ with trivial edge group; denote the corresponding stable letter by $\bar{t}$ (it will also sometimes be convenient to denote it by $\bar{b}_{m+1}$).   We define
$$\Gamma=\pi_1\mathcal{X}$$ and extend $\eta$ to $\eta: G_2\to\Gamma$ by setting $\eta(t)=\bar{t}$. 

\smallskip

{\em Let $F=\langle F_1,F_2\rangle$.  The remainder of this section is devoted to an analysis of the image $\eta(F)=\overline{F}\subseteq\Gamma$. }

\smallskip 

Let $J_0\lhd\Gamma_0$ be a normal, free subgroup of finite index. We use the canonical retraction $\bar{\rho}:\G_0\to\overline{F}_0$
to modify $J_0$, replacing if with $K_0=J_0\cap\bar{\rho}^{-1}(J_0\cap\overline{F}_0)$. The quotient $K_0\backslash T_0$ is a graph $X_0$ with fundamental group $K_0$; $X_0$ may be thought of as a finite-sheeted covering  space of the graph of groups $\mathcal{X}_0$ (this can be made formal, but we will avoid using it explicitly).  There is a natural vertex-transitive left-action of $P=\Gamma_0/K_0$ on $X_0$, in which the stabilizer of each vertex is conjugate to $Q$ (note that $Q$ embeds into $P$ since $Q\cap K_0=1$).  In particular, fixing a base vertex $*$ for $X_0$, we may identify the vertex set of $X_0$ with the coset space $P/Q$. 

There is a minimal $\overline{F}_0$-invariant subtree $T_0^{\overline{F}_0}\subseteq T_0$. Let $Y_0=(\overline{F}_0\cap K_0)\backslash T_0^{\overline{F}_0}$. The inclusion map descends to a combinatorial map $Y_0\to X_0$.  Picking a base vertex in $Y_0$, this map represents the inclusion $\overline{F}_0\cap K_0\to \Gamma_0$.  In fact, this map is an embedding.

\begin{lemma}\label{lem: Embedding}
The graph $Y_0$ is a regular covering of the rose with $m+1$ petals, and the map $\iota:Y_0\to X_0$ is an embedding.
\end{lemma}
\begin{proof}
Note that $\overline{F}_0$ acts freely on $T_0^{\overline{F}_0}$ and transitively on the vertices.  Therefore, the quotient $\overline{F}_0\backslash T_0^{\overline{F}_0}$ is the rose with $m+1$ petals, and $(\overline{F}_0\cap K_0)\backslash T_0^{\overline{F}_0}$ is a regular covering with deck group $R:=\overline{F}_0/(\overline{F}_0\cap K_0)$.

By standard Bass--Serre theory, the fact that $Y_0\to X_0$ is an embedding reduces to the fact that the natural maps
\[
(\overline{F}_0\cap K_0)\backslash \overline{F}_0\to K_0\backslash\Gamma_0/Q
\]
and
\[
(\overline{F}_0\cap K_0)\backslash \overline{F}_0\to K_0\backslash\Gamma_0/\langle \bar{w}\rangle
\]
are injective.  (More exactly, the injectivity of the first map above implies the injectivity of $Y_0\to X_0$ on vertices, and the injectivity of the second map implies the injectivity of $Y_0\to X_0$ on edges.)  Since the first map factors through the second, it is enough to prove that $(\overline{F}_0\cap K_0)\backslash \overline{F}_0\to K_0\backslash\Gamma_0/Q$ is injective.

Suppose therefore that $f,g\in \overline{F}_0$ and $f=kgq$, where $k\in K_0$ and $q\in Q$.  Applying the retraction $\bar{\rho}:\Gamma_0\to\overline{F}_0$, we deduce that $f=\bar{\rho}(k)g$, which implies that 
\[
q= (k^{-1}\bar{\rho}(k))^g~.
\]
But $q$ has finite order and $k^{-1}\bar{\rho}(k)$ lies in the free group $J_0$. Therefore  $q=1$ and
\[
(\overline{F}_0\cap K_0)f=(\overline{F}_0\cap K_0)g
\]
as required.
\end{proof}

We will identify $Y_0$ with its image in $X_0$, and hence we feel free to (without loss of generality) choose $*$ as the base point for $Y_0$.  Fixing a base point allows us to identify the vertices of $Y_0$ with the elements of $R$.

There is a natural retraction $\sigma:\Gamma\to\Gamma_0$ obtained by setting $\sigma(\bar{t})=1$.  The preimage $K=\sigma^{-1}(K_0)$ is a normal, free subgroup of finite index in $\Gamma$ with $\Gamma/K\cong P$. Let $T$ be the Bass--Serre tree of $\mathcal{X}$.  Then $X=K\backslash T$ is a finite graph which, as before, can be thought of as a  regular, finite-sheeted covering space of $\mathcal{X}$ with deck group $P$.

In fact, there is a simple, concrete description of $X$. Consider the graph of groups $\mathcal{Z}$ with a single vertex, labelled by the finite group $Q$, and a single edge, with trivial edge group.  Its fundamental group is $Q*\mathbb{Z}$, which can be identified with $Q*\langle\bar{t}\rangle$, a subgroup of $\Gamma$.  There is an obvious retraction $Q*\langle\bar{t}\rangle\to Q$ obtained by sending $\bar{t}\mapsto 1$, and the kernel is precisely $(Q*\langle\bar{t}\rangle)\cap K$, a normal, torsion-free subgroup of finite index, with quotient group $Q$. The corresponding covering graph of $\mathcal{Z}$ can be constructed as follows.  Let $Z$ be the graph with one vertex and edges $\{e_q\mid q\in Q\}$.  This admits a natural $Q$-action, where  $Q$ acts freely on the edges $e_q$ by left translation, and its fundamental group can be identified with $(Q*\langle\bar{t}\rangle)\cap K$.  

For each coset $pQ\in P/Q$, let $Z^{pQ}$ be a copy of $Z$.  Now $X$ can be constructed as a quotient
\[
X=\left(X_0\sqcup\coprod_{pQ\in P/Q} Z^{pQ}\right)/\sim
\]
where $\sim$ identifies the unique vertex of $Z^{pQ}$ with the vertex of $X_0$ that corresponds to $pQ$ (i.e.\ $p*$).
The group $P$ acts on $X$; the vertex $p*$ is stabilized by $Q^{p^{-1}}$, which acts freely on the edges of $Z^{pQ}$.

The inclusion $Y_0\to X_0$ provides us with a nice geometric representative for the inclusion of $\overline{F}_0\cap K_0$ into $K_0$.   We next extend this to a nice geometric representative for $\overline{F}\cap K$ in $K$.  

Let $W\to Z$ be an immersion (with basepoints) representing $\langle \bar{t}\rangle*\overline{F}_2=\langle \bar{t},\bar{c}_0,\ldots,\bar{c}_{m+1}\rangle$ as a subgroup of the kernel of the natural retraction $Q*\langle\bar{t}\rangle\to Q$. (Note that this immersion exists because $\overline{F}_2\subseteq\langle\!\langle\bar{t}\rangle\!\rangle$.)   Take copies $W^p\equiv W$, one for each $p\in P$, equipped with maps $W^p\to Z^{pQ}$, chosen so that if $pQ=p'Q$ then the following diagram commutes:
\[\xymatrix{
     W^p\ar@{>}[r]\ar@{>}[d]^{\equiv} & Z^{pQ}\ar@{>}[d]^{p'p^{-1}} \\
     W^{p'}\ar@{>}[r] & Z^{pQ}
}\]
where we note that $p'p^{-1}=(p^{-1}p')^{p^{-1}}\in Q^{p^{-1}}$, which acts on $Z^{pQ}$ as remarked above.  Now let
\[
Y= \left(Y_0\sqcup\coprod_{r\in R} W^r\right)/\sim
\]
where $\sim$ identifies the base vertex of $W^r$ with the vertex $r*\in Y_0$.  (Recall that $R=\overline{F}_0/(\overline{F}_0\cap K_0)$, i.e.\ the image of $\overline{F}_0$ in $P$, as in the proof of Lemma \ref{lem: Embedding}.)   The coproduct of the embedding $Y_0\hookrightarrow X_0$ and the immersions $W^r\to Z^{rQ}$ is an immersion $Y\to X$, since adjacent edges of $Y_0$ and $W^r$ map to distinct edges of $X$.  Taking $*\in Y_0$ as a base vertex for $Y$, the immersion $Y\to X$ represents the inclusion of $\overline{F}\cap K$ into $K$.

\begin{lemma}\label{lem: Free product}
If $w\neq_{\widehat{G}}1$ then $\overline{F}=\overline{F}_1*\overline{F}_2$.
\end{lemma}
\begin{proof}
Because free groups are Hopfian, it suffices to prove that $\rk~\overline{F}=\rk~\overline{F}_1+\rk~\overline{F}_2$.   This can be deduced from a computation of the Euler characteristic of $Y$, as follows.

If $d=|R|$, then we have
\begin{eqnarray*}
\chi(Y)&=&\chi(Y_0)+d\chi(W)-d\\
&=& d\left((1-\rk~\overline{F}_0)+(1-(1+\rk~\overline{F}_2))-1\right)\\
&=& d\left(1-(\rk~\overline{F}_0+1+\rk~\overline{F}_2)\right)\\
&=&d(1-(2m+4))~.
\end{eqnarray*}
On the other hand, the fundamental group of $Y$ is $K\cap \overline{F}$, which is of index $d$ in $\overline{F}$.  Therefore
\[
\chi(Y)=d(1-\rk~\overline{F})~.
\]
So $\rk~\overline{F}=2m+4$, which is equal to $\rk~\overline{F}_1+\rk~\overline{F}_2$.
\end{proof}

\subsection*{Malnormality of $\overline{F}$} We shall establish the
malnormality of $\overline{F}$ using the immersion $Y\to X$.  For each left coset $pR\in P/R$,  let $Y^{pR}_0$ be a copy of $Y_0$.  For each coset $pR$ we choose a representative $p_i$ and equip $Y^{pR}_0$ with the inclusion in $X_0$ that is the composition of $p_i$ with the inclusion $Y_0\to X_0$.

Consider
\[
U_0=\coprod_{pR\in P/R} Y^{pR}_0\to X_0~,
\]
the coproduct of the maps described above.     There is a free action of the group $P$ on $U_0$ obtained by insisting that $R$ acts on $Y^R_0$ in the usual way and that $p_i$ takes the base vertex $*_R\in Y^R_0$ to the base vertex $*_{p_iR}\in Y^{p_iR}_0$, and with this definition the map $U_0\to X_0$ is $P$-equivariant. Thus, the vertices of $U_0$ are in bijection with the elements of $P$.  The vertices of $X_0$ are in bijection with $P/Q$, and under this correspondence the map $U_0\to X_0$ on the vertices can be seen as the natural map $P\to P/Q$.

\begin{remark}
Consider the fibre product $U_0\times_{X_0} U_0$.  Note that the map $U_0\to X_0$ represents the family of subgroups $\{\overline{F}_0^{\gamma_i^{-1}}\cap K_0\}$ in $K_0$, where $\gamma_i$ ranges over a set of representatives for $K_0\backslash\Gamma_0/\overline{F}_0$ (which is identified with $P/R$).  Therefore, by Lemmas \ref{lem: Virtual malnormal}, \ref{lem: Fibre product} and \ref{lem: Malnormality of F_0}, the off-diagonal components of $U_0\times_{X_0} U_0$ are simply connected.
\end{remark}

We now consider the same construction for $Y\to X$.  Let $Y^{pR}=Y$ and consider the disjoint union
\[
U=\coprod_{pR\in P/R} Y^{pR}\to X 
\]
where, as before, the map $Y^{pR}\to X$ is the composition of a choice of map $p:X\to X$ with the immersion $Y\to X$.  Alternatively, we can construct $U$ from $U_0$ by attaching copies of $W$ as follows:
\[
U= \left(U_0\sqcup\coprod_{p\in P}W^p\right)/\sim
\]
where $\sim$ identifies the vertex $p*_R\in U_0$ with the base vertex of $W^p$.

The map $U\to X$ represents the family of subgroups $\{\overline{F}^{\gamma_i^{-1}}\cap K\}$ in $K$, where $\gamma_i$ ranges over a set of representatives for $K\backslash\Gamma/\overline{F}=P/R$; therefore, we will be able to prove the malnormality of $\overline{F}$ by considering the fibre product $U\times_X U$.

We can obtain a clearer picture of the map $U\to X$ by first gathering together those copies of $W$ whose images adjoin the same vertex of $X$.  Let
\[
V=\bigcup_{q\in Q} W^q\subseteq U
\]
and note that
\[
U=U_0\cup\bigcup_{p_iQ\in P/Q}p_iV~.
\]
Then $p_iV$ is precisely the preimage of $Z^{p_iQ}\subseteq X$ under the map $U\to X$.

\begin{lemma}\label{lem: Malnormal V}
If $N>6$ then the off-diagonal components of $V\times_Z V$ are simply connected.
\end{lemma}
\begin{proof}
By Lemmas \ref{lem: Virtual malnormal} and \ref{lem: Fibre product}, this is equivalent to the claim that $\langle \bar{t}\rangle*\overline{F}_2\subseteq \langle\bar{t}\rangle*\langle\bar{w}\rangle$ is malnormal in $\langle\bar{t}\rangle*Q$.  This follows from Proposition \ref{prop: m+2-generator malnormal subgroup}, Example \ref{ex: Free product malnormal subgroup} and Lemma \ref{lem: Transitive malnormal}.
\end{proof}

The fibre product $U\times_X U$ decomposes as
\[
U\times_X U= (U_0\times_{X_0}U_0)\cup\coprod_{p_iQ\in P/Q} (p_iV\times_{Z^{p_iQ}} p_iV)
\]
and the diagonal components of $U\times_X U$ consist of precisely the diagonal components of the fibre products on the right hand side of the equation.

\begin{proposition}\label{prop: Malnormality of F}
If $N>6$ and $w\neq_{\widehat{G}}1$ then $\overline{F}$ is malnormal in $\Gamma$.
\end{proposition}
\begin{proof}
By Lemmas \ref{lem: Virtual malnormal} and \ref{lem: Fibre product}, it suffices to show that every off-diagonal component of the fibre product $U\times_X U$ is simply connected.

Suppose therefore that $\delta$ is a geodesic loop in an off-diagonal component of $U$.   The fibre product is equipped with two projections $\pi_1,\pi_2:U\times_X U\to U$ and a $P$-action.  Let $\delta_i=\pi_i\circ\delta$.    Translating by an element of $P$, we may assume that $\delta_1$ is contained in $Y^{R}$.

If $\delta_1\subseteq Y^R_0\subseteq Y^R$ then $\delta_2\subseteq Y^{pR}_0\subseteq Y^{pR}$ for some $p\in P$, so $\delta$ is an essential off-diagonal loop in $U_0\times_{X_0}U_0$, which contradicts the fact that $\overline{F}_0$ is malnormal in $\Gamma$.  Therefore,  $\delta_1$ has a non-trivial subpath contained in $W^r$ for some $r\in R$.  Let $\alpha_1$ be a maximal such subpath, let $\alpha$ be the subpath of $\delta$ with $\pi_1\circ\alpha=\alpha_1$ and let $\alpha_2=\pi_2\circ\alpha$.

The endpoints of $\alpha_1$ lie in $W^r\cap Y^R_0\subseteq Y$; this intersection is a point, and hence $\alpha_1$ is a loop in $W^r$.  Likewise, the endpoints of $\alpha_2$ lie in $W^p\cap Y^{pR}_0$, which is also a point, and so $\alpha_2$ is a loop in $W^p$.  Since they have the same image in $X$ it follows that $p=rq$ for some $q\in Q$.  The loop $r^{-1}\delta$ is then a non-trivial loop in an off-diagonal component of  $V\times_ZV$, which contradicts Lemma \ref{lem: Malnormal V} (since $N>6$).
\end{proof}

\subsection*{Step 4: the end of the proof of Theorem \ref{t:tech}} We take two copies of $G_2$,
distinguishing elements and subgroups of the second by primes, and define 
$G_{w}$ to be the quotient of $G_2*G'_2$  by the relations
\[
\{c_i=b'_i, b_i=c'_i\mid i=0,\ldots,m+1\}~.
\]

If $w=_{\widehat{G}}1$, it is clear that $\widehat{G}_w\cong 1$.

Suppose that $w\neq_{\widehat{G}} 1$. Then $G_w$ is the amalgamated product
\[
G_2*_{F\cong F'} G'_2
\]
where the isomorphism $F\cong F'$ sends $b_i$ to $c'_i$ and $c_i$ to $b'_i$ for $0\leq i\leq {m+1}$. The  map $\eta:G_2\to\Gamma$ constructed at the beginning of Step 3 is injective on $F$, so we obtain an epimorphism
\[
G_w\to\Gamma*_{\overline{F}=\overline{F}'}\Gamma'~.
\]
The latter is an amalgam of virtually free groups along malnormal subgroups, 
and Wise \cite[Theorem 1.3]{wise_residual_2002} proved that
such amalgams are residually finite.  Therefore $\widehat{G}_w\ncong 1$, as required.
\qed

\section{Non-positively curved square complexes}\label{s:npc}

In this section we strengthen Theorem \ref{t:main} by proving that the existence of finite-index subgroups remains undecidable among the fundamental groups of compact, non-positively curved square complexes.  More precisely, we will prove the geometric form of this result stated in the introduction as Theorem \ref{t:covers}.

The arguments in this section are topological in nature and the basic construction is close in spirit to earlier constructions by Kan and Thurston \cite{kan_every_1976}, Leary \cite{leary_metric_2013} and others: the key point in each case is that one replaces a disc in some standard topological construction by a more complicated space that is equally as {\em inessential} as a disc from one point of view but at the same time admits geometric or topological properties that are more desirable from the point of view of the application at hand. In our setting, the standard construction is that of the 2-complex canonically associated to a group presentation, the desirable property is non-positive curvature, and the appropriate notion of {\em inessential} is having a profinitely trivial fundamental group, i.e.~the spaces that replace the disc should have no connected finite-sheeted coverings.

\subsection{An adaptation of the standard 2-complex}
Let
\[
\mathcal{P}\equiv\<a_1,\dots,a_n\mid r_1,\dots,r_m\>
\]
be a finite presentation for a group $G=|\mathcal{P}|$.  
The standard 2-complex $K(\mathcal P)$ with fundamental group $G$ is defined as follows: it has a single vertex, a 1-cell for each generator -- oriented and labelled $a_i$ -- and a 2-cell for each relator, attached along the edge-loop labelled by the word $r_j$, which we may assume to be cyclically reduced.   In what follows, it will be useful to have a name, $R(a_1,\dots,a_n)$ or, more briefly, $R(\underline{a})$, for the 1-skeleton of  $K(\mathcal P)$.

Let $X$ be a compact, non-positively curved square complex with $S=\pi_1X$ infinite but $\widehat{S}\cong 1$ (such as the examples of \cite{burger_finitely_1997} or \cite{wise_complete_2007}) and fix some edge-loop $\gamma:\mathbb{S}^1\to X^{(1)}$ in the 1-skeleton that is a local geodesic in $X$, based at a vertex.

\begin{definition}
Let $S(\mathcal{P})$ be the space obtained by attaching $m$ copies of $X$ to $R(\underline{a})$, with the $j$-th copy attached by a cylinder joining $\gamma$ to the edge-loop in $R(\underline{a})$ labelled $r_j$. More formally, writing $\rho_j: \mathbb{S}^1\to R(\underline{a})$  for this last loop, we define $\sim$ to be the equivalence relation on 
\[
R(\underline{a})\coprod \big(\mathbb{S}^1\times [0,1]\big)\times\{1,\dots,m\} \coprod X\times \{1,\dots,m\}
\]
defined by
\[
\forall t\in \mathbb{S}^1\, \forall j\in\{1,\dots,m\}\ :\  \rho_j(t)\sim (t,0,j) {\rm{~and~}} (t,1,j)\sim (\gamma(t),j),
\]
and define $S(\mathcal{P})$ to be the quotient space.  Define $G_S:=\pi_1S(\mathcal{P})$.
\end{definition}

\begin{remarks}
\begin{enumerate}
\item For any fixed choice of $\gamma$, the construction of $S(\mathcal{P})$ from $\mathcal{P}$ is algorithmic.

\item There is a continuous map $\rho:S(\mathcal{P})\to K(\mathcal{P})$ that is the identity on $R(\underline{a})$, sends each copy of $X$ to a point in the interior of the corresponding 2-cell of $K(\mathcal{P})$, and maps the
interior of each attaching cylinder homeomorphically to the interior of a punctured 2-cell. This map induces  epimorphisms $\rho_*:G_S\to G$ and $\widehat{\rho}_*:\widehat{G}_S\to \widehat{G}$.
\end{enumerate}
\end{remarks}

\begin{lemma}\label{lem: Profinite isomorphism}
The map $\widehat{\rho}_*:\widehat{G}_S\to \widehat{G}$ is an isomorphism.
\end{lemma}
\begin{proof}
It is enough to show that any homomorphism $f$ from $G_S$ to a finite group factors through $\rho_*$. By construction, $S$ has no finite quotients, so $f(S_j)=1$ where $S_j\cong S$ is the fundamental group of the copy of $X$ in $S(\mathcal{P})$ indexed by $j\in\{1,\dots,m\}.$
\end{proof}

\begin{lemma}\label{lem: NPC square complex}
For any finite group presentation $\mathcal{P}$, the space $S(\mathcal{P})$ 
has the structure of  a finite, non-positively curved square complex.
\end{lemma}
\begin{proof}
Let $k$ be the length of $\gamma$. We scale $R(a_1,\dots,a_n)$ by a factor of $k$ and subdivide each edge into $k$ pieces of length $1$. For $j=1,\dots,m$ we take a copy of $X$ scaled by a factor of the word-length of $r_j$, subdivided in the natural way so that it is a (unit) square complex. The attaching maps in the definition of  $S(\mathcal{P})$ are then length-preserving, so if the connecting cylinders are subdivided into squares in the obvious manner, $S(\mathcal{P})$ becomes a non-positively curved square complex \cite[Proposition II.11.6]{bridson_metric_1999}.
\end{proof}

Together, these lemmas establish the following proposition, which reduces Theorem \ref{t:covers} to
Theorem \ref{t:main}.

\begin{proposition}\label{lem: P squared}
There is an algorithm that takes as input a finite group presentation $\mathcal{P}$ for a group $G$ and outputs a compact, non-positively curved square complex $S(\mathcal{P})$ with fundamental group $G_S$ such that
\[
\widehat{G}_S\cong\widehat{G}~.
\]
\end{proposition}

\begin{remark} A simple combinatorial check will determine if a finite square complex satisfies the link condition, i.e.~supports a metric of non-positive curvature. Thus, the class of such 2-complexes is recursive.
\end{remark}

\subsection{Largeness}

A group is called  \emph{large} (or \emph{as large as a free group}, in the original terminology of Pride \cite{pride_concept_1980}), if it has a subgroup of finite index
that maps surjectively to a non-abelian free group.  Largeness is related to the existence of finite quotients
by the following elementary observation.

\begin{lemma}\label{lem: Large triple}
A group $G$ has a non-trivial finite quotient
if and only if  $G*G*G$ is large. 
\end{lemma}

\begin{proof} If $G$ maps onto a non-trivial finite group $Q$, then $G*G*G$ maps onto $Q\ast Q\ast Q$.
The kernel of any homomorphism $Q\ast Q\ast Q\to Q$ that restricts to an isomorphism on each of the
free factors is non-abelian and free of finite index, and a subgroup of finite index in $G*G*G$ maps onto it.  Conversely, if $G$ can only map trivially to a finite group, then the same is true of $G\ast G\ast G$; so it is not large.
\end{proof}

 Combining Lemma \ref{lem: Large triple} with Theorem \ref{t:covers}, we see that largeness is undecidable, even among the fundamental groups of non-positively curved square complexes.

\begin{corollary}\label{cor: NPC largeness}
There is a recursive sequence of finite, non-positively curved square complexes $X_n$ such that:
\begin{enumerate}
\item for each $n\in\N$, $X_n$ has a proper connected finite-sheeted covering space if and only if $\pi_1X_n$ is large; 
\item the set of natural numbers
\[
\{n\in\N\mid \pi_1X_n\mathrm{~is~large}\}
\]
is recursively enumerable but not recursive.
\end{enumerate}
In particular, there is no algorithm to determine whether or not the fundamental group of a finite, non-positively curved square complex is large.
\end{corollary}

\subsection{Biautomatic groups}

Fundamental groups of compact, non-positively curved square complexes are biautomatic \cite{gersten_small_1990} (see also \cite{niblo_geometry_1998}). There is an algorithm to determine if a biautomatic group is trivial, but Theorem \ref{t:covers} tells us that there is no algorithm to determine if it is profinitely trivial.

\begin{corollary}
There is no algorithm that, given a biautomatic group $G$, can determine whether or not $G$ has a proper subgroup of finite index. Nor is there an algorithm that can determine whether or not $G$ is large.
\end{corollary}

\section{Profinite Rank}
\label{s:profinite}

By definition, the {\em profinite rank} of a group $G$, denoted by $\hat{d}(G)$, is the minimum number of elements needed to generate $\wh{G}$ as a topological group.

\subsection{A profinite Grushko lemma}
We want to show that there is no algorithm that can determine the profinite rank of a hyperbolic group. For this we shall use the following analogue of Grushko's theorem; we make no claim that the constant $\frac{59}{60}$ is sharp.

\begin{lemma}\label{lem: Profinite Grushko} 
Let $G$ be a group with $\widehat{G}\ncong 1$. Then $\hat d (\bigast_{i=1}^n G) \ge \frac{59}{60}n$.
\end{lemma}
\begin{proof}
If $G$ maps onto a finite cyclic group $\Z/p$, then $L_n:=\bigast_{i=1}^n G$ and   its profinite completion  map  onto $(\Z/p)^n$, and therefore
 require at least $n$ generators.

Suppose, then, that  $G$ maps onto a non-trivial finite perfect group $S$.  Let $Q_n:=\bigast_{i=1}^nS$ and let $\pi:Q_n\to S$ be a homomorphism that restricts to an isomorphism on each free factor.  The kernel $\ker\pi$ acts freely on the Bass--Serre tree for $Q_n$ (since all of the torsion of $Q_n$ is conjugate into one of the free factors) and hence $\ker\pi$ is a free group; its rank is $r:=(n-1)(|S|-1)$, as can be calculated using rational Euler characteristic.

Let $K_n<L_n$ be the inverse image of $\ker\pi$. Then $K_n$ is normal,   maps onto a free group of rank $r$, and $L_n/K_n\cong S$. We fix an epimorphism $K_n\to (\Z/2)^{r}=:A$ and induce this to a homomorphism $\Phi:L_n\to A\wr S$. The image of $K_n$ under this map lies in the base of the wreath product, where it projects onto each $A$ summand; thus it is an elementary 2-group of rank at least $r$.

By the Nielsen--Schreier formula, if $\Phi(L_n)$ has rank $\delta$ then $\Phi(K_n)$, which has index at most $|S|$, has rank at most $|S|(\delta -1) +1$. Thus  
\[
(n-1)(|S|-1) \le |S| (\delta -1) +1~,
\]
whence
\[
\hat d(L_n) \ge \delta \geq \left(\frac{|S|-1}{|S|}\right)n~.
\]
But $S$ is perfect and non-trivial, so $|S|\ge 60$.
\end{proof}

\subsection{Profinite rank of hyperbolic groups}

We shall appeal to the following version of the Rips construction.

\begin{theorem}\label{thm: Rips--Wise}
There is an algorithm that takes as input a finite presentation for a group $G$ and outputs a finite presentation for a residually finite, torsion-free, hyperbolic group $\Gamma$ such that there exists a short exact sequence
\[
1\to N\to \Gamma\to G\to 1
\]
where $N$ is a 2-generator group.
\end{theorem}
\begin{proof}
Rips showed how to construct such a short exact sequence with $\Gamma$ satisfying the $C'(1/6)$ small-cancellation condition \cite{rips_subgroups_1982}.  Wise proved that such groups are fundamental groups of compact, non-positively curved cube complexes \cite{wise_cubulating_2004}.  By Agol's theorem \cite{agol_virtual_2013}, it follows that $\Gamma$ is virtually special and, in particular, residually finite.
\end{proof}

We can now prove part (\ref{i: profinite rank}) of Theorem \ref{t:hyp}.  Note that the examples constructed are residually finite.

\begin{theorem}\label{thm: Hyperbolic profinite rank}
Fix any $d_0>2$. There is a recursive sequence of torsion-free, residually finite, hyperbolic groups $\Gamma_n$ with the property that:
\begin{enumerate}
\item for any $n\in\N$, $\hat{d}(\Gamma_n)<d_0 \Leftrightarrow \hat{d}({\Gamma}_n)=2$; and
\item the set of natural numbers
\[
\{n\in\N\mid \hat{d}(\Gamma_n)\geq d_0\}
\]
is recursively enumerable but not recursive.
\end{enumerate}
In particular, there is no algorithm that can decide whether or not the profinite completion of a torsion-free, residually finite, hyperbolic group can be generated (topologically) by a set of cardinality less than $d_0$.
\end{theorem}
\begin{proof}
Let $G_n$ be a recursive sequence of finitely presented groups such that the set of natural numbers $\{n\in\N\mid \widehat{G}_n\ncong 1\}$ is recursively enumerable but not recursive.  Let $M\geq\frac{60}{59}d_0$ and, for each $n$, let $G'_n$ be a free product of $M$ copies of $G_n$. Then either $\widehat{G}'_n\cong 1$ or $\hat{d}(G'_n)\geq d_0$ by Lemma \ref{lem: Profinite Grushko}.

Apply Theorem \ref{thm: Rips--Wise} to obtain short exact sequences
\[
1\to N_n\to\Gamma_n\to G'_n\to 1
\]
with each $N_n$ a 2-generator group.  Note that since $\G_n$ is residually-finite but not cyclic, $\hat{d}(\Gamma_n)\geq 2$. 

If $\hat{d}(\Gamma_n)<d_0$ then $\hat{d}(G'_n)<d_0$, so $\widehat{G}'_n\cong 1$ and $\widehat{N}_n$ surjects $\widehat{\Gamma}_n$, whence $\hat{d}(\Gamma_n)=2$.  This proves (1).  Item (2) follows, because $\hat{d}(\Gamma_n)\geq d_0$ if and only if $\widehat{G}_n\ncong 1$.
\end{proof}

\section{Undecidable properties of hyperbolic groups}\label{s:hyp}

In this section we prove the remaining parts of Theorem \ref{t:hyp}. We also prove that either every hyperbolic group is residually finite, or else there is no algorithm to decide which hyperbolic groups have a finite quotient. All of these things will be proved by combining our previous results with the following refinement of the Rips construction \cite{rips_subgroups_1982}, which is due to  Belagradek and Osin \cite{belegradek_rips_2008}. 

\begin{theorem}[Belegradek--Osin, \cite{belegradek_rips_2008}]\label{thm: Belegradek--Osin}
There is an algorithm that takes as input a finite presentation for a non-elementary hyperbolic group $H$ and a finite presentation for a group $G$ and outputs a presentation for a hyperbolic group $\Gamma$ that fits into a short exact sequence
\[
1\to N\to \Gamma\to G\to 1 
\]
such that $N$ is isomorphic to a quotient group of $H$.  Furthermore, if $H$ and $G$ are torsion-free then $\Gamma$ can also be taken to be torsion-free.
\end{theorem}
\begin{proof}
The only point that is not addressed directly by Belegradek and Osin is the fact that the construction can be made algorithmic, but it is tacitly implied in Corollary 3.8 of \cite{belegradek_rips_2008}. Indeed, since the class of hyperbolic groups is recursively enumerable \cite{papasoglu_algorithm_1996}, a naive search will eventually find a hyperbolic group $\Gamma$ and a homomorphism $H\to\Gamma$ whose image is normal with quotient  isomorphic to $G$.

In the torsion-free case, one needs the well known fact that the class of torsion-free hyperbolic groups is also recursively enumerable (see, for instance, the proof of Theorem III.$\Gamma$.3.2 in \cite{bridson_metric_1999}).
\end{proof}

\subsection{Largeness and virtual first Betti number}

Parts (\ref{i: largeness}) and (\ref{i: vb}) of Theorem \ref{t:hyp} follow from the next theorem. 

\begin{theorem}\label{thm: Hyperbolic largeness and vb_1}
There is a recursive sequence of finite presentations for torsion-free, hyperbolic groups $\G_n$ such that:
\begin{enumerate}
\item for each $n\in\N$,
\[
vb_1(\G_n)>0 \Leftrightarrow vb_1(\G_n)=\infty \Leftrightarrow \G_n\mathrm{~is~large}~;
\]
and
\item the set of natural numbers
\[
\{n\in\N\mid \G_n\mathrm{~is~large}\}
\]
is recursively enumerable but not recursive.
\end{enumerate}
In particular, for any $1\leq d\leq\infty$, there is no algorithm that determines whether or not a given torsion-free hyperbolic group $\Gamma$ has $vb_1(\Gamma)\geq d$; likewise, there is no algorithm that determines whether or not a given torsion-free hyperbolic group is large.  
\end{theorem}
\begin{proof} Let $G_n$ be the sequence of fundamental groups of the square complexes produced by Corollary \ref{cor: NPC largeness}; note that as the fundamental groups of aspherical spaces, the $G_n$ are torsion-free.  Let $N_n<\Gamma_n$ be the pair of groups obtained by applying the algorithm of Theorem \ref{thm: Belegradek--Osin} to $G_n$, with $H$ a fixed torsion-free, non-elementary hyperbolic group with Property (T); torsion-free uniform lattices in ${\rm{Sp}}(n,1)$ provide explicit examples.

We have the following chain of implications.
\[
vb_1(G_n)>0\Rightarrow\G_n\mathrm{~is~large}\Rightarrow vb_1(\G_n)=\infty\Rightarrow vb_1(\G_n)>0
\]
The first implication follows from part (1) of Corollary \ref{cor: NPC largeness}, and the other implications are trivial.

To prove (1) and (2), it therefore suffices to show that $vb_1(\Gamma_n)>0$ implies that $vb_1(G_n)>0$.  Suppose, therefore, that $K<\G_n$ is a subgroup of finite index that admits a surjection $f:K\to \Z$.  Property (T) is inherited by quotients and subgroups of finite index, so the abelianization of $N_n\cap K$ is finite. Therefore, $f(N_n\cap K)=1$ and so $K/(K\cap N_n)$ surjects $\Z$. But $K/(K\cap N_n)$ has finite index in $G_n$, so $vb_1(G_n)>0$ as required.
\end{proof}

\subsection{Linear representations}

In this section we make use of known examples of torsion-free, non-elementary hyperbolic groups that admit no infinite linear representation to establish parts (\ref{i: finite linear rep}) and (\ref{i: finite linear rep fixed k}) of Theorem \ref{t:hyp}.  As M. Kapovich showed in \cite[Theorem 8.1]{kapovich_representations_2005}, the existence of such examples can be proved using the work of Corlette \cite{corlette_archimedean_1992} and Gromov--Schoen \cite{gromov_harmonic_1992} on (archimedean and non-archimedean) super-rigidity for lattices in ${\rm{Sp}}(n,1)$. 

\begin{theorem}[\cite{kapovich_representations_2005}]\label{thm: Non-linear hyperbolic group}
There exists a torsion-free, non-elementary hyperbolic group $H$ with the property that, for any field $k$, every finite-dimensional representation of $G$ over $k$ has finite image.
\end{theorem}
\begin{proof}
The statement of this theorem is the same as \cite[Theorem 8.1]{kapovich_representations_2005}, with the additional stipulation that the group $H$ is torsion-free.  Following Kapovich, we start with a uniform lattice $\Gamma$ in the isometry group of quaternionic hyperbolic space.  By Selberg's Lemma, we may assume that $\Gamma$ is torsion free.  We then take $H$ (which is $G$ in Kapovich's notation) to be any infinite small-cancellation quotient of $\Gamma$.  As Kapovich explains, the group $H$ then has no infinite linear representations over any field.

For a suitable choice of small-cancellation quotient, any torsion in $H$ is the image of torsion in $\Gamma$.  (For instance, this follows from \cite[Lemma 6.3]{osin_small_2010}, which even deals with the relatively hyperbolic setting.)  Such a choice of $H$ is therefore torsion-free. 
\end{proof}

The following theorem covers parts (\ref{i: finite linear rep}) and (\ref{i: finite linear rep fixed k}) of Theorem \ref{t:hyp}.

\begin{theorem}\label{thm: Hyperbolic linear representations}
Fix any infinite field $k$. There is a recursive sequence of torsion-free hyperbolic groups $\Gamma_n$ with the property that:
\begin{enumerate}
\item for any $n\in\N$, $\Gamma_n$ has a finite-dimensional representation over $k$ with infinite image if and only if $\Gamma_n$ has a finite-dimensional representation over some field with infinite image; and
\item the set of $n\in\N$ such that $\Gamma_n$ has a finite-dimensional representation over $k$ with infinite image is recursively enumerable but not recursive.
\end{enumerate}
\end{theorem}
\begin{proof}
Let $X_n$ be the sequence of square complexes output by Corollary \ref{cor: NPC largeness} and let $G_n=\pi_1X_n$. Finitely generated linear groups are residually finite, so for any infinite field $k$, $G_n$ has a finite-dimensional representation over $k$ with infinite image if and only if $G_n$ is large; furthermore, the set of natural numbers $n$ such that $G_n$ has such a representation is recursively enumerable but not recursive.

Let $H$ be the torsion-free, non-elementary hyperbolic group of Theorem \ref{thm: Non-linear hyperbolic group}.  For each $n$, let $\Gamma_n$ be the torsion-free hyperbolic group that is the output of the algorithm of Theorem \ref{thm: Belegradek--Osin} with input $G_n$ and $H$.

The result now follows from the claim that, for any field $k$, $\Gamma_n$ has a finite-dimensional representation over $k$ with infinite image if and only if $G_n$ does.  Indeed, if $G_n$ has such a representation then $\Gamma_n$ clearly does.  Conversely, suppose that $f:\Gamma_n\to {\rm{GL}}(m,k)$ has infinite image.  If $N$ is the kernel of the map $\Gamma_n\to G_n$ then, because $N$ is a quotient of $H$, it follows that $f(N)$ is finite. Because $f(\Gamma_n)$ is residually finite, there exists a proper subgroup $K$ of finite index in $f(\Gamma_n)$ such that $K\cap f(N)=1$.  Then $L=f^{-1}(K)$ is a subgroup of finite index in $\Gamma_n$ with an infinite representation $f|_L$ over $k$, and $f|_L(L\cap N)=1$.  Therefore, $f|_L$ factors through the restriction to $L$ of the map $\Gamma_n\to G_n$. It follows that $G_n$ has a subgroup of finite index with an infinite representation over $k$, and so $G_n$ also has such a representation.
\end{proof}

\subsection{Profinite undecidability in the hyperbolic case}\label{ss: RF or undecidability}

We finish with the following conjecture.

\begin{conjecture}\label{conj: Hyperbolic undecidability}
There is no algorithm that can determine whether or not a given hyperbolic group $\Gamma$ has $\widehat{\Gamma}\cong1$.
\end{conjecture}

Since the triviality problem is solvable for hyperbolic groups, the above conjecture is false if every non-trivial hyperbolic group $\Gamma$ has $\widehat{\Gamma}\ncong 1$.  In fact, I. Kapovich and Wise proved that every non-trivial (torsion-free) hyperbolic group $\Gamma$ has $\widehat{\Gamma}\ncong 1$ if and only if every (torsion-free) hyperbolic group is residually finite \cite{kapovich_equivalence_2000}.  Conjecture \ref{conj: Hyperbolic undecidability} therefore implies the well known conjecture that there exists a non-residually finite hyperbolic group \cite[Question 1.15]{bestvina_questions_????}.  In fact, our final theorem shows that the two conjectures are equivalent (even in the torsion-free case).

\begin{theorem}\label{t:mainHyp} 
The following statements are equivalent.
\begin{enumerate}
\item Every non-trivial (torsion-free) hyperbolic group has a proper subgroup of finite index.
\item There is an algorithm that, given a finite presentation of a  (torsion-free) hyperbolic
group, will  determine whether or not the profinite completion of that group is trivial.
\end{enumerate}
\end{theorem}
\begin{proof}
There is an algorithm that can decide if a given hyperbolic group is trivial,
and if (1) holds then (2) reduces to checking if the given group is trivial.  For the converse, suppose that there exists a non-trivial hyperbolic group $H_0$ with $\widehat{H}_0=1$.  Clearly $H_0$ is non-elementary.  Let $G_n$ be a sequence of (torsion-free) groups that witnesses the undecidability in Theorem \ref{t:covers}, let $\Gamma_n$ be the sequence of hyperbolic groups obtained by applying Theorem \ref{thm: Belegradek--Osin} to $G_n$ with $H=H_0$, and note that $\widehat{H}=1$ implies $\widehat{\Gamma}_n\cong\widehat{G}_n$.  It is a feature of Theorem \ref{thm: Belegradek--Osin} that if $H_0$ is torsion-free then so are the groups $\Gamma_n$. 
\end{proof}

\bibliographystyle{plain}

\end{document}

%% file: core.pdf_tex
\begingroup%
  \makeatletter%
  \providecommand\color[2][]{%
    \errmessage{(Inkscape) Color is used for the text in Inkscape, but the package 'color.sty' is not loaded}%
    \renewcommand\color[2][]{}%
  }%
  \providecommand\transparent[1]{%
    \errmessage{(Inkscape) Transparency is used (non-zero) for the text in Inkscape, but the package 'transparent.sty' is not loaded}%
    \renewcommand\transparent[1]{}%
  }%
  \providecommand\rotatebox[2]{#2}%
  \ifx\svgwidth\undefined%
    \setlength{\unitlength}{184.28665161bp}%
    \ifx\svgscale\undefined%
      \relax%
    \else%
      \setlength{\unitlength}{\unitlength * \real{\svgscale}}%
    \fi%
  \else%
    \setlength{\unitlength}{\svgwidth}%
  \fi%
  \global\let\svgwidth\undefined%
  \global\let\svgscale\undefined%
  \makeatother%
  \begin{picture}(1,1.16177548)%
    \put(0,0){\includegraphics[width=\unitlength]{core.pdf}}%
    \put(0.52060584,0.01023796){\color[rgb]{0,0,0}\makebox(0,0)[lb]{\smash{$e_0$}}}%
    \put(0.58696144,0.79986961){\color[rgb]{0,0,0}\makebox(0,0)[lb]{\smash{$e_1$}}}%
    \put(0.8722905,0.79986961){\color[rgb]{0,0,0}\makebox(0,0)[lb]{\smash{$e_2$}}}%
    \put(-0.00360342,0.79986961){\color[rgb]{0,0,0}\makebox(0,0)[lb]{\smash{$e_{-1}$}}}%
  \end{picture}%
\endgroup%

%% file: fibre.pdf_tex
\begingroup%
  \makeatletter%
  \providecommand\color[2][]{%
    \errmessage{(Inkscape) Color is used for the text in Inkscape, but the package 'color.sty' is not loaded}%
    \renewcommand\color[2][]{}%
  }%
  \providecommand\transparent[1]{%
    \errmessage{(Inkscape) Transparency is used (non-zero) for the text in Inkscape, but the package 'transparent.sty' is not loaded}%
    \renewcommand\transparent[1]{}%
  }%
  \providecommand\rotatebox[2]{#2}%
  \ifx\svgwidth\undefined%
    \setlength{\unitlength}{768.96546597bp}%
    \ifx\svgscale\undefined%
      \relax%
    \else%
      \setlength{\unitlength}{\unitlength * \real{\svgscale}}%
    \fi%
  \else%
    \setlength{\unitlength}{\svgwidth}%
  \fi%
  \global\let\svgwidth\undefined%
  \global\let\svgscale\undefined%
  \makeatother%
  \begin{picture}(1,0.55998656)%
    \put(0,0){\includegraphics[width=\unitlength]{fibre.pdf}}%
    \put(0.15118828,0.02301015){\color[rgb]{0,0,0}\makebox(0,0)[lb]{\smash{$e_0$}}}%
    \put(0.16709075,0.21224956){\color[rgb]{0,0,0}\makebox(0,0)[lb]{\smash{$e_1$}}}%
    \put(0.23547137,0.21224956){\color[rgb]{0,0,0}\makebox(0,0)[lb]{\smash{$e_2$}}}%
    \put(0.02555875,0.21224956){\color[rgb]{0,0,0}\makebox(0,0)[lb]{\smash{$e_{-1}$}}}%
    \put(-0.00086358,0.10673666){\color[rgb]{0,0,0}\makebox(0,0)[lb]{\smash{$Y_0$}}}%
    \put(0.59831716,0.4505519){\color[rgb]{0,0,0}\makebox(0,0)[lb]{\smash{$e_0$}}}%
    \put(0.39719607,0.47899143){\color[rgb]{0,0,0}\makebox(0,0)[lb]{\smash{$e_1$}}}%
    \put(0.39473479,0.54870885){\color[rgb]{0,0,0}\makebox(0,0)[lb]{\smash{$e_2$}}}%
    \put(0.39284996,0.36766846){\color[rgb]{0,0,0}\makebox(0,0)[lb]{\smash{$e_{-1}$}}}%
    \put(0.49221355,0.52435719){\color[rgb]{0,0,0}\makebox(0,0)[lb]{\smash{$Y_0$}}}%
    \put(0.96939374,0.45392532){\color[rgb]{0,0,0}\makebox(0,0)[lb]{\smash{$e_1$}}}%
    \put(0.76827264,0.48236485){\color[rgb]{0,0,0}\makebox(0,0)[lb]{\smash{$e_2$}}}%
    \put(0.76581136,0.55208227){\color[rgb]{0,0,0}\makebox(0,0)[lb]{\smash{$e_3$}}}%
    \put(0.76392653,0.37104188){\color[rgb]{0,0,0}\makebox(0,0)[lb]{\smash{$e_0$}}}%
    \put(0.86606807,0.51929709){\color[rgb]{0,0,0}\makebox(0,0)[lb]{\smash{$Y_1$}}}%
    \put(0.76074605,0.12614382){\color[rgb]{0,0,0}\makebox(0,0)[lb]{\smash{$e_0$}}}%
    \put(0.76270684,0.27467109){\color[rgb]{0,0,0}\makebox(0,0)[lb]{\smash{$e_2$}}}%
    \put(0.86674272,0.23305674){\color[rgb]{0,0,0}\makebox(0,0)[lb]{\smash{$e_1$}}}%
    \put(0.32575614,0.00245358){\color[rgb]{0,0,0}\makebox(0,0)[lb]{\smash{$Y_0\times_{X}Y_0$}}}%
    \put(0.70028531,0.00576804){\color[rgb]{0,0,0}\makebox(0,0)[lb]{\smash{$Y_0\times_{X}Y_1$}}}%
    \put(0.28414178,0.19144239){\color[rgb]{0,0,0}\makebox(0,0)[lb]{\smash{$e_{-1}$}}}%
    \put(0.39958143,0.24346033){\color[rgb]{0,0,0}\makebox(0,0)[lb]{\smash{$e_1$}}}%
    \put(0.48180996,0.24258627){\color[rgb]{0,0,0}\makebox(0,0)[lb]{\smash{$e_2$}}}%
    \put(0.54504387,0.15874156){\color[rgb]{0,0,0}\makebox(0,0)[lb]{\smash{$e_0$}}}%
  \end{picture}%
\endgroup%